\documentclass[a4paper,12pt, reqno]{amsart}
\usepackage{latexsym}
\usepackage{color,dsfont}
\usepackage[usenames,dvipsnames]{xcolor}
\usepackage{amssymb,amsfonts,amsmath,mathrsfs}
\newcommand{\eps}{\varepsilon}

\newtheorem{theorem}{Theorem}[section]
\newtheorem{proposition}[theorem]{Proposition}
\newtheorem{lemma}[theorem]{Lemma}

\newtheorem{remark}[theorem]{Remark}

\theoremstyle{definition}

\numberwithin{equation}{section}

\title[Twisted second moment of Dirichlet $L$-functions]{Breaking the $\frac{1}{2}$-barrier for the twisted second moment of Dirichlet $L$-functions}
\author{H. M. Bui, Kyle Pratt, Nicolas Robles and Alexandru Zaharescu}
\address{School of Mathematics, University of Manchester, Manchester M13 9PL, UK}
\email{hung.bui@manchester.ac.uk}
\address{Department of Mathematics, University of Illinois, 1409 West Green Street, Urbana, IL 61801, USA}
\email{kpratt4@illinois.edu}
\email{nirobles@illinois.edu}
\address{Department of Mathematics, University of Illinois, 1409 West Green Street, Urbana, IL 61801, USA
and Simion Stoilow Institute of Mathematics of the Romanian Academy, P.O. Box 1-764, RO-014700 Bucharest,
Romania}
\email{zaharesc@illinois.edu}

\allowdisplaybreaks

\begin{document} 

\maketitle

\begin{abstract}
We study the second moment of Dirichlet $L$-functions to a large prime modulus $q$ twisted by the square of an arbitrary Dirichlet polynomial. We break the $\frac{1}{2}$-barrier in this problem, and obtain an asymptotic formula provided that the length of the Dirichlet polynomial is less than $q^{51/101} = q^{1/2 +1/202}$. As an application, we obtain an upper bound of the correct order of magnitude for the third moment of Dirichlet $L$-functions. We give further results when the coefficients of the Dirichlet polynomial are more specialized. 
\end{abstract}

\section{Introduction}

We study the mean square of the product of Dirichlet $L$-functions with arbitrary Dirichlet polynomials. The central problem is to obtain an asymptotic formula for
\begin{align}\label{eq: intro main quantity of interest}
\sideset{}{^*}\sum_{\chi (\textrm{mod}\ q)} \left|L \left( \tfrac{1}{2},\chi \right)\right|^2 \bigg|\sum_{a \leq q^\kappa} \frac{\alpha_a \chi(a)}{\sqrt{a}} \bigg|^2,
\end{align}
where $\sum^*$ denotes summation over all primitive characters $\chi$ modulo $q$, the coefficients $\alpha_a \ll a^\varepsilon$ are arbitrary, and $0 < \kappa < 1$. The asymptotic evaluation when $\kappa < 1/2$ was established by Iwaniec and Sarnak [\textbf{\ref{IS}}]. In this regime the main term comes from the ``diagonal'' contribution, and the off-diagonal terms contribute only to the error.

In this work we obtain an asymptotic expression for \eqref{eq: intro main quantity of interest} with $\kappa$ going beyond the $\frac{1}{2}$-barrier. In this larger regime some off-diagonal terms contribute to the main term, and evaluating this contribution, as well as bounding the error terms, is considerably more difficult. Duke, Friedlander and Iwaniec [\textbf{\ref{DFI1}}] proved that the quantity in \eqref{eq: intro main quantity of interest} may be bounded by $O_\eps(q^{1+\varepsilon})$ for some $\kappa > 1/2$, but their proof does not extend to give an asymptotic formula. Very recently, Conrey, Iwaniec and Soundararajan [\textbf{\ref{CIS}}] applied their asymptotic large sieve and studied \eqref{eq: intro main quantity of interest} with an additional averaging over the modulus $q$.

The analogous problem to \eqref{eq: intro main quantity of interest} for the Riemann zeta-function was studied by Bettin, Chandee and Radziwi\l\l \ [\textbf{\ref{BCR}}], who broke the $\frac{1}{2}$-barrier for an arbitrary Dirichlet polynomial. Our work is inspired by their beautiful paper, but there are significant differences between the family of primitive Dirichlet $L$-functions in the $q$-aspect and the Riemann zeta-function in the $t$-aspect. These differences usually make the family of Dirichlet $L$-functions more difficult to work with. In fact, Bettin, Chandee and Radziwi\l\l \ mentioned the work of Duke, Friedlander and Iwaniec [\textbf{\ref{DFI1}}], and said: ``Our proof of Theorem 1 would not extend to give an asymptotic in this case, and additional input is needed.''

It turns out that it is more convenient to work with a more general version of \eqref{eq: intro main quantity of interest} by introducing ``shifts''. We let $\alpha,\beta \in \mathbb{C}$ satisfy $\text{Re}(\alpha),\text{Re}(\beta) \ll (\log q)^{-1}$ and, for our later application\footnote{See Lemma \ref{lemmathird}}, $\text{Im}(\alpha),\text{Im}(\beta) \ll \log q$. Furthermore, we treat the characters $\chi$ according to their parity to ensure that the $L$-functions under consideration have the same gamma factors in their functional equations. Let $\varphi^+(q)$ be the number of even primitive characters $\chi$ modulo $q$, and let $\sum^+$ denote summation over all these characters. In this work we deal exclusively with even Dirichlet characters, but our arguments go through identically for odd characters. For fixed $0 < \kappa < 1$, we study
\begin{equation*}
I_{\alpha,\beta}=\frac{1}{\varphi^+(q)}\ \sideset{}{^+}\sum_{\chi(\textrm{mod}\ q)} L(\tfrac 12+\alpha,\chi)L(\tfrac12+\beta,\overline{\chi})\big|A(\chi)\big|^2,
\end{equation*}
where
\begin{align*}
A(\chi) &=\sum_{a\leq q^\kappa}\frac{\alpha_a\chi(a)}{\sqrt{a}},
\end{align*}
and $\alpha_a$ is an arbitrary sequence of complex numbers satisfying $\alpha_a \ll_\varepsilon a^\varepsilon$.

For technical convenience we assume that $q$ is prime throughout the paper. Note that in this case, the number of primitive characters is $\varphi^*(q)=q-2$, and we have $\varphi^+(q)=(q-3)/2$. It is likely that our methods could be adapted to the case of general $q$ with more effort.

The following is our main theorem.
\begin{theorem}\label{mthm}
Suppose that $q$ is prime. Let $\alpha,\beta \in \mathbb{C}$ satisfy 
\begin{align*}
|\textup{Re}(\alpha)|,|\textup{Re}(\beta)| \ll (\log q)^{-1}\qquad\emph{and}\qquad|\textup{Im}(\alpha)|,|\textup{Im}(\beta)| \ll \log q.
\end{align*}
Suppose that $\kappa< 1/2 + 1/202$. Then
\begin{align}\label{mainfml}
\!\!I_{\alpha,\beta}&=\zeta(1+\alpha+\beta)\sum_{\substack{da,db\leq q^\kappa\\(a,b)=1}}\frac{\alpha_{da}\overline{\alpha_{db}}}{da^{1+\beta}b^{1+\alpha}}\nonumber\\
&\qquad+\Big(\frac{q}{\pi}\Big)^{-(\alpha+\beta)}\frac{\Gamma(\frac{1/2- \alpha}{2})\Gamma(\frac{1/2-\beta}{2})}{\Gamma(\frac{1/2+\alpha}{2})\Gamma(\frac{1/2+\beta}{2})}\zeta(1-\alpha-\beta)\sum_{\substack{da,db\leq q^\kappa\\(a,b)=1}}\frac{\alpha_{da}\overline{\alpha_{db}}}{da^{1-\alpha}b^{1-\beta}}+O(q^{-\delta_0})
\end{align}
for some $\delta_0>0$.
\end{theorem}

The use of Theorem 2 of [\textbf{\ref{DFI2}}] and our Proposition \ref{prop2} below also suffices to break the $\frac12$-barrier. With those results, Theorem \ref{mthm} holds provided that $\kappa<1/2+1/526$.


The range of $\kappa$ can be enlarged if we know more about the Dirichlet polynomial $A(\chi)$. Let $\gamma$ be a smooth function supported in $[1,2]$ such that $\gamma^{(j)}\ll_j q^\eps$ for any fixed $j\geq0$. Suppose that $\alpha=\eta*\lambda$, where $\eta_{a_1},\lambda_{a_2}$ are two sequences of complex numbers supported on $[1,A_1]$ and $[1,A_2]$, respectively, with $A_1=q^{\kappa_1}$, $A_2=q^{\kappa_2}$ and $\kappa=\kappa_1+\kappa_2$, and satisfy $\eta_a,\lambda_a\ll a^\eps$. Friedlander and Iwaniec [\textbf{\ref{FI}}] showed that
\begin{align*}
\frac{1}{\varphi^*(q)}\ \sideset{}{^*}\sum_{\chi(\text{mod}\ q)}\bigg|\sum_{n}\frac{\chi(n)}{\sqrt{n}}\gamma\Big(\frac nN\Big)\bigg|^2\big|A(\chi)\big|^2 \ll_\eps q^\eps+q^{-3/4+5\kappa/4+\eps}\big(A_1+A_2\big)^{1/4},
\end{align*}
if $N\ll q^{1/2+\eps}$. When $A_1\asymp A_2$, their result gives an upper bound of $O_\eps(q^\varepsilon)$ provided that $\kappa<1/2+1/22$.

\begin{theorem}\label{thm2}
Assume as above and suppose further that $9\kappa+\max\{\kappa_1,\kappa_2\}<5$. Then \eqref{mainfml} holds
for some $\delta_0>0$.
\end{theorem}

If $\kappa_1=\kappa_2$, then Theorem \ref{thm2} allows us to take $\kappa<1/2+1/38$.

Another case of special interest is when $\alpha=\eta*\lambda$ with $\eta_a$ being smooth coefficients up to $q^{1/2+\eps}$ and $\lambda_a$ being arbitrary and as long as possible. This can be viewed as an analogue of Hough' result [\textbf{\ref{H}}; Theorem 4] (see also [\textbf{\ref{Z}}; Theorem 1.1]), which gives an asymptotic formula for the fourth moment of Dirichlet $L$-functions twisted by the square of a Dirichlet polynomial of length less than $q^{1/32}$. 

Suppose that 
\[
\eta_{a_1}=\eta\Big(\frac{a_1}{A_1}\Big),
\] 
where $\eta$ is a smooth function supported in $[1,2]$ such that $\eta^{(j)}\ll_j q^{\eps}$ for any fixed $j\geq0$. If $N,A_1\ll q^{1/2+\eps}$, then Watt [\textbf{\ref{W}}] proved that
\begin{align*}
\frac{1}{\varphi^*(q)}\ \sideset{}{^*}\sum_{\chi(\text{mod}\ q)}\bigg|\sum_{n}\frac{\chi(n)}{\sqrt{n}}\gamma\Big(\frac nN\Big)\bigg|^2\big|A(\chi)\big|^2 \ll_\eps q^\eps+q^{\vartheta-1/2+\eps}A_2^2,
\end{align*}
where $\vartheta=7/64$. This yields an upper bound of $O_\eps(q^\varepsilon)$ provided that $\kappa_2<1/4-\vartheta/2$.

\begin{theorem}\label{thm3}
Assume as above and suppose further that $\kappa_2<1/14-\vartheta/7$ with $\vartheta=7/64$. Then \eqref{mainfml} holds
for some $\delta_0>0$.
\end{theorem}

As an application of Theorem \ref{mthm}, we obtain the order of magnitude of the third moment of Dirichlet $L$-functions.
\begin{theorem}\label{thm: third moment}
Suppose that $q$ is prime. Then
\begin{align}\label{eq: third moment}
\frac{1}{\varphi^*(q)}\ \sideset{}{^*}\sum_{\chi(\emph{mod}\ q)} \left|L \left( \tfrac{1}{2},\chi\right) \right|^3 \asymp (\log q)^{9/4}.
\end{align}
\end{theorem}
We remark that it is also possible to obtain upper bounds on all moments below the fourth by adapting the work of Radziwi\l\l\, and Soundararajan [\textbf{\ref{RS1}}] and Hough [\textbf{\ref{H}}].

An important application of the twisted second moment of Dirichlet $L$-functions is in regard to non-vanishing of Dirichlet $L$-functions at the central point $s=1/2$. It is widely believed that $L(1/2,\chi)\ne 0$ for all primitive characters $\chi$. At least $34\%$ of Dirichlet $L$-functions in the family of primitive characters, to a large modulus $q$, are known to not vanish at the central point [\textbf{\ref{B}}] (see also [\textbf{\ref{IS}}]). Our results here, together with the mollifier method, may be used to give a slight improvement of such a result. However, we note that in the case of large prime $q$, Khan and Ngo [\textbf{\ref{KN}}], using the ``twisted" mollifier introduced by [\textbf{\ref{S}}] and [\textbf{\ref{MV}}],  have obtained a non-vanishing proportion of $3/8$. Improving that using the usual mollifier would require a mollifier with $\kappa>3/5$, which seems out of reach of the current techniques.

\subsection{Comparison with the $t$-aspect analogue}

We close the introduction with a brief discussion of the main differences between our work and the work of Bettin, Chandee and Radziwi\l\l's on the twisted second moment of the Riemann zeta-function [\textbf{\ref{BCR}}].

In [\textbf{\ref{BCR}}], after applying the approximate functional equation and dyadic decompositions, the main object to study is of the form
\[
\frac{1}{\sqrt{ABMN}}\mathop{\sum \sum \sum\sum}_{\substack{a \asymp A, b \asymp B \\ m \asymp M, n \asymp N,MN\ll T^{1+\eps}}} \alpha_a \beta_b\, \widehat{W}\Big(T\log\frac{am}{bn}\Big),
\]
where $W$, say, is some compactly supported smooth function satisfying $W^{(j)}\ll_j T^{\eps}$ for any fixed $j\geq0$. The diagonal contribution $am=bn$ may be extracted and is fairly easy to understand. For the remaining terms, write $am-bn=r$ with $r\ne0$. The rapid decay of the Fourier transform implies that the contribution of the terms with $|r|>T^{-1+\eps}\sqrt{ABMN}$ is negligible. The off-diagonal contribution is then essentially 
\[
\frac{1}{\sqrt{ABMN}}\sum_{0<|r|\leq T^{-1+\eps}\sqrt{ABMN}}\mathop{\sum \sum \sum\sum}_{\substack{am-bn=r\\a \asymp A, b \asymp B \\ m \asymp M, n \asymp N,MN\ll T^{1+\eps}}} \alpha_a \beta_b.
\]
In particular, this is null unless $AM\asymp BN$. Writing $am-bn=r$ as a congruence condition modulo $b$ and applying the Poisson summation formula transform the above expression to an exponential sum roughly of the form
\begin{equation}\label{BCRexp}
\frac{\sqrt{MN}}{\sqrt{AB}(AM+BN)}\mathop{\sum \sum}_{\substack{0<|r|\leq T^{-1+\eps}\sqrt{ABMN}\\|g|\ll T^{\eps}(AM+BN)/MN}}\,\mathop{\sum \sum}_{\substack{a \asymp A, b \asymp B}} \alpha_a \beta_b\,e\Big(\frac{-rg\overline{a}}{b}\Big).
\end{equation}
Trivially, this is bounded by $T^{-1+\eps}AB\ll T^{2\kappa-1+\eps}$. So any extra power saving is sufficient to break the $\frac12$-barrier. Bettin, Chandee and Radziwi\l\l\ [\textbf{\ref{BCR}}] gained this by utilizing an estimate for sums of Kloosterman fractions in [\textbf{\ref{BC}}] (or, [\textbf{\ref{DFI2}}]).

In our situation, we use the approximate functional equation and apply character orthogonality, as well as dyadic decompositions, to reduce the problem to understanding sums of the form
\begin{align*}
\frac{1}{\sqrt{ABMN}}\mathop{\sum \sum \sum\sum}_{\substack{am \equiv\pm bn (\textrm{mod}\ q)\\a \asymp A, b \asymp B \\ m \asymp M, n \asymp N,MN\ll q^{1+\eps}}} \alpha_a \beta_b.
\end{align*}
Diagonal main terms arise from $am = bn$. We set this contribution aside, and study the remaining terms. 
At this point, we can, as above, write the congruence condition modulo $q$ as $am\mp bn=qr$ with $r\ne0$, and switch to a congruence condition modulo $b$. The Poisson summation formula then leads to an exponential sum roughly of the form
\begin{equation}\label{ourexp}
\frac{\sqrt{MN}}{\sqrt{AB}(AM+BN)}\mathop{\sum \sum}_{\substack{0<|r|\ll q^{-1}(AM+BN)\\|g|\ll q^{\eps}(AM+BN)/MN}}\,\mathop{\sum \sum}_{\substack{a \asymp A, b \asymp B}} \alpha_a \beta_b\,e\Big(\frac{-qrg\overline{a}}{b}\Big).
\end{equation}

The results on cancellation in sums of Kloosterman fractions [\textbf{\ref{DFI2}}, \textbf{\ref{BC}}] are still applicable, but only work in the ``balanced'' case when $AM$ and $BN$ are more or less of the same size. Roughly speaking, the $t$-aspect averaging yields the constraint $AM\asymp BN$ and ensures that $r,g\ll T^{\kappa-1/2+\eps}$ in \eqref{BCRexp}. In our situation \eqref{ourexp}, we lack the condition $AM \asymp BN$, and therefore the ranges of summation of $r$ and $g$ can be as large as $q^{\kappa+\eps}$. The trivial bound, which can be $O_\eps(q^{2\kappa-1/2+\eps})$, is worse in our case as well. So a different method is required for the ``unbalanced'' regime when $AM$ and $BN$ are of rather different sizes. Furthermore, it is impossible to ignore the contribution of these terms. This contribution is genuinely large, and, as it turns out, will cancel out with the contribution from the principal character modulo $q$. A similar phenomenon already arose in Young's work on the fourth moment of Dirichlet $L$-functions [\textbf{\ref{Y}}]. We remark that it is possible to show that $I_{\alpha,\beta}\ll_\eps q^\varepsilon$ for some $\kappa > 1/2$ without considering the unbalanced case [\textbf{\ref{DFI1}}].

We begin the treatment of the unbalanced regime by applying the Poisson summation formula to introduce exponential phases. The zero frequency cancels out with the contribution from the principal character modulo $q$. We bound the contribution of the non-zero frequencies with a delicate argument involving the additive large sieve inequality. The phases of the exponentials are rational fractions, and we divide these fractions into two classes: ``good'' fractions and ``bad'' fractions. The good fractions are far apart from each other, and we can immediately apply the additive large sieve inequality to get a saving. The bad fractions can be close together, which weakens the large sieve inequality, but we still obtain a saving since there are comparatively few of these bad fractions.

Lastly, we remark that with our two different methods for the two different regimes, the critical ranges of summation in \eqref{ourexp} are when $A\asymp B\asymp N\asymp q^\kappa$ and $M\asymp q^{2-3\kappa}$. This explains why the ranges of $\kappa$ for the asymptotic formula in our theorems are slightly smaller than the corresponding ranges for the upper bound.  

\begin{remark}
\emph{Throughout the paper $\varepsilon$ denotes an arbitrarily small positive number whose value may change from one line to the next.}
\end{remark}

\section{Initial manipulations}\label{initial}

We start by recalling the orthogonality property of characters and the approximate functional equation.

\begin{lemma}[Orthogonality]\label{ortho}
For $(mn,q)=1$ we have
\[
\sideset{}{^+}\sum_{\chi(\emph{mod}\ q)} \chi(m)\overline{\chi}(n)=\frac 12\sum_{\substack{d|q\\d|(m\pm n)}}\mu\Big(\frac qd\Big)\varphi(d)=\frac 12\sum_{q|(m\pm n)}\varphi(q)-1.
\]
\end{lemma}
\begin{proof}
The proof is standard. See, for example, [\textbf{\ref{IS}}; (3.1) and (3.2)].
\end{proof}

\begin{lemma}[Approximate functional equation]\label{fe}
Let $\chi$ be an even primitive character and let $G(s)$ be an even entire function of rapid decay in any fixed strip $|\emph{Re}(s)| \leq C$ satisfying $G(0) = 1$. Let
\[
X_\pm(s)=G(s)\frac{\Gamma(\frac{1/2\pm \alpha+s}{2})\Gamma(\frac{1/2\pm\beta+s}{2})}{\Gamma(\frac{1/2+\alpha}{2})\Gamma(\frac{1/2+\beta}{2})}
\]
and
\begin{equation}\label{Vpm}
V_{\pm}(x)=\frac{1}{2\pi i}\int_{(\eps)}X_\pm(s)x^{-s}\frac{ds}{s}.
\end{equation}
Then
\begin{align*}
L(\tfrac 12+\alpha,\chi)L(\tfrac12+\beta,\overline{\chi})&=\sum_{m,n\geq1}\frac{\chi(m)\overline{\chi}(n)}{m^{1/2+\alpha}n^{1/2+\beta}}V_+\Big(\frac{\pi mn}{q}\Big)\\
&\qquad\qquad+\Big(\frac{q}{\pi}\Big)^{-(\alpha+\beta)}\sum_{m,n\geq1}\frac{\chi(m)\overline{\chi}(n)}{m^{1/2-\beta}n^{1/2-\alpha}}V_-\Big(\frac{\pi mn}{q}\Big).
\end{align*}
\end{lemma}
\begin{proof}
The proof is standard and can be easily derived following Theorem 5.3 of [\textbf{\ref{IK}}].
\end{proof}
\begin{remark}\label{rmkV} \
\begin{itemize}
\item \emph{It is convenient to prescribe certain conditions on the function $G$. To be precise, we assume $G(s)$ is invariant under the transformations $\alpha\rightarrow-\beta$, $\beta\rightarrow-\alpha$, and vanishes at $s=\pm(\alpha+\beta)/2$. An admissible choice of $G$ is $$G(s)=\frac{(\frac{\alpha+\beta}{2})^2-s^2}{(\frac{\alpha+\beta}{2})^2}\,e^{s^2},$$ but there is no need to specify a particular function $G$.}
\item \emph{If $|\text{Re}(\alpha)|,|\text{Re}(\beta)|\ll (\log q)^{-1}$ and $|\text{Im}(\alpha)|,|\text{Im}(\beta)|\leq T$, then Stirling's approximation gives $$x^jV_{\pm}^{(j)}(x)\ll_{j,C} (1+|x|/T)^{-C}$$ for any fixed $j\geq0$ and $C>0$.}
\end{itemize}
\end{remark}


From Lemma \ref{fe} and Lemma \ref{ortho} we get
\begin{eqnarray*}
I_{\alpha,\beta}&=&\sum_{\substack{a,b\leq q^\kappa\\m,n\geq1}}\frac{\alpha_a\overline{\alpha_b}}{\sqrt{ab}m^{1/2+\alpha}n^{1/2+\beta}}V_{+}\Big(\frac{\pi mn}{q}\Big)\frac{1}{\varphi^+(q)}\ \sideset{}{^+}\sum_{\chi(\textrm{mod}\ q)} \chi(am)\overline{\chi}(bn)\nonumber\\
&&\quad\quad+\Big(\frac{q}{\pi}\Big)^{-(\alpha+\beta)}\sum_{\substack{a,b\leq q^\kappa\\m,n\geq1}}\frac{\alpha_a\overline{\alpha_b}}{\sqrt{ab}m^{1/2-\beta}n^{1/2-\alpha}}V_{-}\Big(\frac{\pi mn}{q}\Big)\frac{1}{\varphi^+(q)}\ \sideset{}{^+}\sum_{\chi(\textrm{mod}\ q)} \chi(am)\overline{\chi}(bn)\\
&=&J_{\alpha,\beta}^{+}+\Big(\frac{q}{\pi}\Big)^{-(\alpha+\beta)}J_{-\beta,-\alpha}^{-},
\end{eqnarray*}
where 
\begin{align*}
J_{\alpha,\beta}^{+}&=\frac{\varphi(q)}{2\varphi^+(q)}\sum_{\substack{a,b\leq q^\kappa\\am\equiv \pm bn(\text{mod}\ q)\\(mn,q)=1}}\frac{\alpha_a\overline{\alpha_b}}{\sqrt{ab}m^{1/2+\alpha}n^{1/2+\beta}}V_{+}\Big(\frac{\pi mn}{q}\Big)\\
&\qquad\qquad\qquad\qquad\qquad\qquad-\frac{1}{\varphi^+(q)}\sum_{\substack{a,b\leq q^\kappa\\(mn,q)=1}}\frac{\alpha_a\overline{\alpha_b}}{\sqrt{ab}m^{1/2+\alpha}n^{1/2+\beta}}V_{+}\Big(\frac{\pi mn}{q}\Big),
\end{align*}
and $J^-$ is the same but with $V_-$ in place of $V_+$. Note from Remark \ref{rmkV} that the condition $(mn,q)=1$ may be omitted with the cost of an error of size $O_\eps(q^{\kappa-3/2+\eps})$. The same error applies when we replace $\varphi(q)$ by $q$ and $\varphi^+(q)$ by $q/2$. So, up to an error term of size $O_\eps(q^{\kappa-3/2+\eps})$, we have
\begin{align*}
J_{\alpha,\beta}^{+}&=\sum_{\substack{a,b\leq q^\kappa\\am\equiv\pm bn(\text{mod}\ q)}}\frac{\alpha_a\overline{\alpha_b}}{\sqrt{ab}m^{1/2+\alpha}n^{1/2+\beta}}V_{+}\Big(\frac{\pi mn}{q}\Big)-\frac{2}{q}\sum_{\substack{a,b\leq q^\kappa\\m,n\geq1}}\frac{\alpha_a\overline{\alpha_b}}{\sqrt{ab}m^{1/2+\alpha}n^{1/2+\beta}}V_{+}\Big(\frac{\pi mn}{q}\Big)\\
&=\mathcal{M}_{\alpha,\beta}^{+}+S_{\alpha,\beta}^{+}-\frac{2}{q}\sum_{\substack{a,b\leq q^\kappa\\m,n\geq1}}\frac{\alpha_a\overline{\alpha_b}}{\sqrt{ab}m^{1/2+\alpha}n^{1/2+\beta}}V_{+}\Big(\frac{\pi mn}{q}\Big),
\end{align*}
where $\mathcal{M}_{\alpha,\beta}^{+}$ and $S_{\alpha,\beta}^{+}$ are the contributions from the diagonal terms $am=bn$ and the off-diagonal terms $am\ne bn$ in the first sum, respectively. We similarly define $\mathcal{M}_{-\beta,-\alpha}^{-}$ and $S_{-\beta,-\alpha}^{-}$.

Much of this paper is spent studying the off-diagonal terms $S^\pm$. We shall complete the proofs of Theorem \ref{mthm} and Theorems \ref{thm2}, \ref{thm3} in Sections \ref{smthm} and \ref{sthm2}. We just finish this section with a partial evaluation of $\mathcal{M}^{\pm}$. Replacing $a,b$ by $da,db$ with $(a,b)=1$, we see that $m=bn'$ and $n=an'$ for some $n'\in\mathbb{N}$. Hence
\begin{align}\label{M+}
\mathcal{M}_{\alpha,\beta}^{+}&=\sum_{\substack{da,db\leq q^\kappa\\(a,b)=1}}\frac{\alpha_{da}\overline{\alpha_{db}}}{da^{1+\beta}b^{1+\alpha}}\sum_{n\geq1}\frac{1}{n^{1+\alpha+\beta}}V_{+}\Big(\frac{\pi abn^2}{q}\Big)\nonumber\\
&=\sum_{\substack{da,db\leq q^\kappa\\(a,b)=1}}\frac{\alpha_{da}\overline{\alpha_{db}}}{da^{1+\beta}b^{1+\alpha}}\frac{1}{2\pi i}\int_{(\eps)}X_+(s)\Big(\frac{q}{\pi ab}\Big)^s\zeta(1+\alpha+\beta+2s)\frac{ds}{s},
\end{align}
and a similar expression holds for $\mathcal{M}^-$.

\section{Main propositions}

In this section we focus on the sum
\[
\sum_{\substack{am\equiv\pm bn(\text{mod}\ q)\\am\ne bn}}\alpha_a\beta_b
\]
over dyadic intervals. Theorem \ref{mthm} is deduced from the following two propositions. The first result essentially treats the case when $m$ and $n$ are close, while the second one deals with the case when $m$ and $n$ are far apart. 

\begin{proposition}\label{prop1}
Let $A,B,M,N\geq1$, and let $\alpha_a,\beta_b$ be two sequences of complex numbers supported on $[A,2A]$ and $[B,2B]$ satisfying $\alpha_a\ll A^\eps, \beta_b\ll B^\eps$. Let $W_1$ and $W_2$ be smooth functions supported in $[1,2]$ such that $W_1^{(j)},W_2^{(j)}\ll_j q^{\eps}$ for any fixed $j\geq0$. Let
\[
\mathcal S =\frac{1}{\sqrt{ABMN}}\sum_{\substack{am\equiv\pm bn(\emph{mod}\ q)\\am\ne bn}}\alpha_a\beta_bW_1\Big(\frac{m}{M}\Big)W_2\Big(\frac{n}{N}\Big).
\]
Then
\begin{equation}\label{prop1-1}
\mathcal{S}\ll_\eps q^{-1+\eps}\sqrt{ABMN}.
\end{equation}

If $ABMN\gg q^{2-\eps}$, then 
\begin{equation}\label{prop1-2}
\mathcal S= \mathcal M_1^++\mathcal M_1^-+\mathcal E ,
\end{equation}
where
\begin{equation}\label{mainM2}
\mathcal{M}_1^\pm=\frac{1}{\sqrt{ABMN}}\sum_{\substack{d\geq 1\\r\ne0}}\sum_{(a,b)=1}\alpha_{da}\beta_{db}\int  W_1\Big(\frac{bx}{M}\Big)W_2\Big(\frac{\pm(abx-qr)}{bN}\Big)dx
\end{equation}
and
\begin{align}\label{bdE1}
\mathcal{E}&\ll_\eps q^{-17/20+\eps}(AB)^{-3/20}(AM+BN)^{17/10}(A+B)^{1/4}(MN)^{-17/20}\nonumber\\
&\qquad\qquad +q^{-1+\eps}(AB)^{-1/8}(AM+BN)^2(A+B)^{1/8}(MN)^{-1}.
\end{align}

\end{proposition}


 \begin{proof}
To prove \eqref{prop1-1}, we write $am\mp bn=qr$. Then $$0< |r|\leq R=4(AM+BN)q^{-1},$$ and
\begin{equation}\label{555}
\mathcal{S}=\frac{1}{\sqrt{ABMN}}\sum_{0<|r|\leq R}\sum_{am\mp bn=qr}\alpha_a\beta_bW_1\Big(\frac{m}{M}\Big)W_2\Big(\frac{n}{N}\Big).
\end{equation}
We can take the innermost sum over all $a$ and $m$, and the sums over $b$ and $n$ being over all divisors of $(am-qr)$. So, by symmetry,
\begin{equation*}
\mathcal{S}\ll_\eps \frac{q^\eps R}{\sqrt{ABMN}}\min\{AM,BN\}\ll_\eps q^{-1+\eps}\sqrt{ABMN},
\end{equation*}
and we obtain \eqref{prop1-1}.



We next prove \eqref{prop1-2}. We proceed from \eqref{555}. Without loss of generality, assume that $AM\ll BN$ (consequently, we also have $qR\ll BN$). This simplifies the decision on which variable to eliminate. We will eliminate $n$ in this case, so we first introduce Mellin inversion to write
\begin{align*}
W_2\Big(\frac{n}{N}\Big)&=W_2\Big(\frac{\pm(am-qr)}{bN}\Big)\\
&=\frac{1}{(2\pi i)^4}\int_{(\varepsilon)}\int_{(\varepsilon)}\int_{(\varepsilon)}\int_{(\varepsilon)}\widehat{f_\pm}(u,v,w,z)a^{-u}b^{-v}m^{-w}r^{-z}dudvdwdz.
\end{align*} 
Note that by integration by parts $j$ times on each variable of the expression
\[
\widehat{f_\pm}(u,v,w,z)=\int_{0}^{\infty}\int_{0}^{\infty}\int_{0}^{\infty}\int_{0}^{\infty}a^{u-1}b^{v-1}m^{w-1}r^{z-1}W_2\Big(\frac{\pm(am-qr)}{bN}\Big)dadbdmdr,
\]
 we have
\begin{equation}\label{fhat}
\widehat{f_\pm}(u,v,w,z)\ll_{\eps,j} A^{\text{Re}(u)}B^{\text{Re}(v)}M^{\text{Re}(w)}R^{\text{Re}(z)}\Big(q^{-\eps} \big(1+|u|\big)\big(1+|v|\big)\big(1+|w|\big)\big(1+|z|\big)\Big)^{-j}
\end{equation}
uniformly for $\text{Re}(u),\text{Re}(v),\text{Re}(w),\text{Re}(z)\geq \eps$, and for any fixed $j\geq 0$. Hence we may restrict the $u,v,w,z$-integrals to
\[
|u|,|v|,|w|,|z|\ll q^\eps.
\]

Let $d=(a,b)$ (note that this implies $d|r$). We can now remove $n$ by writing $am\mp bn=qr$ as $m\equiv q(r/d)\overline{a/d}(\textrm{mod}\ b/d)$. The Poisson summation formula  then yields
\begin{eqnarray}\label{1}
\!\!\!\!\!\sum_{\substack{m,n\\am\mp bn=qr}}m^{-w}W_1\Big(\frac{m}{M}\Big)&&=\sum_{m\equiv q(r/d)\overline{a/d}(\textrm{mod}\ b/d)}m^{-w}W_1\Big(\frac{m}{M}\Big)\nonumber\\
&&=\sum_{g\in\mathbb{Z}}e\Big(\frac{-q(r/d)g\overline{a/d}}{b/d}\Big)\int \Big(\frac{bx}{d}\Big)^{-w}W_1\Big(\frac{bx}{dM}\Big)e(gx)dx.
\end{eqnarray}
Note that the integral is over $x\asymp dM/B\asymp dMN/(AM+BN)$.

For the term $g=0$, we fold back the Mellin inversion and get the terms $\mathcal{M}_1^\pm$ in \eqref{mainM2}. The restriction $d|r|\leq R$ may be removed due to the support of $W$. For the terms $g\ne0$, integration by parts $j$ times implies that
\begin{eqnarray*}
\int \Big(\frac{bx}{d}\Big)^{-w}W_1\Big(\frac{bx}{dM}\Big)e(gx)dx&\ll_j&\frac{dM^{1-\text{Re}(w)}}{B}\Big(\frac{(1+|w|)B}{gdM}\Big)^j
\end{eqnarray*}
for any fixed $j\geq 0$. So we may restrict the sum in \eqref{1} to $0<|g|\leq G/d$, where 
\[
G=\frac{ q^\eps B}{M}\asymp\frac{ q^\eps(AM+BN)}{MN}.
\]
Hence,
\begin{equation}\label{contS}
\mathcal{S}=\mathcal{M}_1^++\mathcal{M}_1^-+ \mathcal{E}+O_C\big(q^{-C}\big)
\end{equation}
with
\begin{align}\label{boundE1}
\mathcal{E}\ll &\frac{1}{\sqrt{ABMN}}\sum_{d\leq R}\int_{(\varepsilon)}\int_{(\varepsilon)}\int_{(\varepsilon)}\int_{(\varepsilon)}\int_{x\asymp dMN/(AM+BN)}\nonumber \\
&\qquad\qquad x^{-\text{Re}(w)}\big|\widehat{f_\pm}(u,v,w,z)\big|\big|Z_d(x)\big|dxdudvdwdz,
\end{align}
where
\begin{equation*}
Z_d(x)=d^{-(u+v+z)}\sum_{\substack{0<|r|\leq R/d\\0<|g|\leq G/d}}\sum_{(a,b)=1}a^{-u}b^{-(v+w)}r^{-z}\alpha_{da}\beta_{db}\,e\Big(\frac{-qrg\overline{a}}{b}+gx\Big)W_1\Big(\frac{bx}{M}\Big).
\end{equation*}

The above expression is in the form of Theorem 1 of [\textbf{\ref{BC}}]. 
We apply this and get
\begin{align*}
Z_d(x)&\ll_\eps q^\eps\Big(\frac{ABRG}{d^4}\Big)^{1/2}\Big(1+\frac{qRG}{AB}\Big)^{1/2}\\
&\qquad\quad\bigg(\Big(\frac{ABRG}{d^4}\Big)^{7/20}\Big(\frac Ad+\frac Bd\Big)^{1/4}+\Big(\frac{AB}{d^2}\Big)^{3/8}\Big(\frac{RG}{d^2}\Big)^{1/2}\Big(\frac Ad+\frac Bd\Big)^{1/8}\bigg)\\
& \ll_\eps d^{-2}q^{-1/2+\eps}\frac{(AM+BN)^2}{MN}\\
&\qquad\quad\Big(\frac{(AB)^{7/20}(AM+BN)^{7/10}(A+B)^{1/4}}{d^{33/20}(qMN)^{7/20}}+\frac{(AB)^{3/8}(AM+BN)(A+B)^{1/8}}{d^{15/8}(qMN)^{1/2}}\Big).
\end{align*}
Plugging this into \eqref{boundE1} and using the bound in \eqref{fhat} we obtain \eqref{bdE1}.
\end{proof}

\begin{proposition}\label{prop2}
Assume the conditions of Proposition \ref{prop1} and that $BM\ll q^{1-\eps}$, $A\ll N$, $B\gg M$. Then
\[
\mathcal S= \mathcal M_2+\mathcal E ,
\]
where
\begin{equation}\label{mainM}
\mathcal{M}_2=\frac{2}{\sqrt{ABMN}}\sum_{a,b}\sum_m\alpha_a\beta_bW_1\Big(\frac{m}{M}\Big)\int W_2\Big(\frac{qx}{N}\Big)dx
\end{equation}
and
\begin{eqnarray*}
\mathcal E\ll_\varepsilon  \Big(\frac{q}{BM}\Big)^{-1/6+\eps}+q^\eps\Big(\sqrt{\frac AN}+\sqrt{\frac MB}\Big).
\end{eqnarray*}
\end{proposition}
\begin{proof}
Applying the Poisson summation formula over $n$ gives
\begin{align*}
\mathcal S &=\frac{1}{\sqrt{ABMN}}\sum_{a,b}\sum_{m}\alpha_a\beta_bW_1\Big(\frac{m}{M}\Big)\sum_{n\equiv\pm am\overline{b}(\text{mod}\ q)}W_2\Big(\frac{n}{N}\Big)\\
&=\frac{\sqrt{N}}{q\sqrt{ABM}}\sum_{a,b}\sum_{m}\sum_{h\in\mathbb{Z}}\alpha_a\beta_b\, e\Big(\frac{\pm ahm\overline{b}}{q}\Big)W_1\Big(\frac{m}{M}\Big)\widehat{W_2}\Big(\frac{hN}{q}\Big).
\end{align*}
The contribution of the term $h=0$ corresponds to the term $\mathcal{M}_2$ in \eqref{mainM}. For $h\ne0$ the rapid decay of the Fourier transform means we may restrict the sum over $h$ to  $0<|h|\leq H$, where 
\[
H=\frac{q^{1+\eps}}{N}.
\]
Thus we have
\[
\mathcal{S}=\mathcal{M}_2+\mathcal{E}+O_C\big(q^{-C}\big),
\]
where
\begin{align}\label{Z1}
\mathcal{E}&=\frac{\sqrt{N}}{q\sqrt{ABM}}\sum_{a,b}\sum_m\sum_{0<|h|\leq H}\alpha_a\beta_b\, e\Big(\frac{\pm ahm\overline{b}}{q}\Big)W_1\Big(\frac{m}{M}\Big)\widehat{W_2}\Big(\frac{hN}{q}\Big)\nonumber\\
&=\frac{\sqrt{N}}{q\sqrt{ABM}}\sum_{l}\sum_{u\in\mathcal{U}_l}\nu_l(u)\sum_{a}\sum_{0<|h|\leq H}\alpha_a\, e\Big(\frac{\pm ahu}{q}\Big)\widehat{W_2}\Big(\frac{hN}{q}\Big).
\end{align}
Here
\begin{equation}\label{condition}\mathcal{U}_l=\big\{0<u<q: \exists\, B/l\leq b\leq 2B/l,\, M/l\leq m\leq 2M/l,\, (b,m)=1,\,m\overline{b}\equiv u(\text{mod}\ q)\big\}\end{equation}
and $$\nu_l(u)=\sum_{\substack{m\overline{b}\equiv u(\text{mod}\ q)\\(b,m)=1}}\beta_{lb}W_1\Big(\frac{lm}{M}\Big).$$ Note that if \begin{equation*}(b_1,m_1)=(b_2,m_2)=1\qquad \text{and}\qquad m_1\overline{b_1}\equiv m_2\overline{b_2}\equiv u(\text{mod}\ q),
\end{equation*} then $b_1m_2\equiv b_2m_1(\text{mod}\ q)$. Since $BM\ll q^{1-\eps}$, it follows that $b_1m_2=b_2m_1$, and hence $b_1=b_2$ and $m_1=m_2$. So given $u\in\mathcal{U}_l$, the existence of the pair $(b,m)$ in \eqref{condition} is unique. In particular we get
\begin{equation}\label{bdU}
\#\,\mathcal{U}_l\ll \frac{BM}{l^2}\qquad\text{and}\qquad \nu_l(u)\ll_\eps q^\eps.
\end{equation}

Given $0<X<q$, we call a pair $(u_1,u_2)$ ``$(l,X)$-bad'' if $u_1,u_2\in\mathcal{U}_l$ and 
\begin{align*}
u_1- u_2\equiv s(\text{mod}\ q)
\end{align*}
with $0<|s|\leq X$. Define $\mathcal{U}_l^{\text{bad}}\subset \mathcal{U}_l$ to be the set consisting of all $u$'s which belong to at least one such pair, and set $\mathcal{U}_l^\text{good}=\mathcal{U}_l\backslash\mathcal{U}_l^\text{bad}.$

 Let
\begin{align*}
Y_l=\#\big\{B/l\leq b_1,b_2\leq 2B/l,&\,M/l\leq m_1,m_2\leq 2M/l,\,0<|s|\leq X:\\
&(b_1,m_1)=(b_2,m_2)=1\ \text{and}\ m_1\overline{b_1}- m_2\overline{b_2}\equiv s(\text{mod}\ q)\big\}.
\end{align*}
We have
\begin{align*}
m_1\overline{b_1}- m_2\overline{b_2}\equiv s(\text{mod}\ q) &\Longleftrightarrow sb_1b_2+ b_1m_2-b_2m_1 \equiv 0(\text{mod}\ q)\\
&\Longleftrightarrow (sb_1-m_1)(sb_2+m_2)\equiv -m_1m_2(\text{mod}\ q).
\end{align*}
The facts that given $m_1, m_2$ and the product $(sb_1-m_1)(sb_2+m_2)$, the number of triples $(b_1,b_2,s)$ is $\ll_\eps q^\eps$, and that $(sb_1-m_1)(sb_2+m_2)\ll X^2(B/l)^2$ lead to
\begin{align*}
Y_l&\ll_\eps q^\eps \Big(\frac Ml\Big)^2\Big(1+\frac{X^2B^2}{l^2q}\Big)\ll_\eps q^\eps\Big(\frac{M^2}{l^2}+\frac{X^2B^2M^2}{l^4q}\Big).
\end{align*}
The same bound, hence, applies to $\#\,\mathcal{U}_l^\text{bad}$, and so
\begin{equation}\label{bdU1}
\sum_{u\in\mathcal{U}_l^\text{bad}}|\nu_l(u)|^2\ll_\eps q^\eps\Big(\frac{M^2}{l^2}+\frac{X^2B^2M^2}{l^4q}\Big).
\end{equation}

We write \eqref{Z1} as
\[
\mathcal{E}=\frac{\sqrt{N}}{q\sqrt{ABM}}\bigg(\sum_l\sum_{u\in\mathcal{U}_l^\text{bad}}+\sum_l\sum_{u\in\mathcal{U}_l^\text{good}}\bigg)=\frac{\sqrt{N}}{q\sqrt{ABM}}\big(Z_\text{bad}+Z_\text{good}\big),
\]
say. For $Z_\text{bad}$, we apply Cauchy's inequality, \eqref{bdU1} and the additive large sieve inequality (see [\textbf{\ref{mont}}; Corollary 2.2] or [\textbf{\ref{IK}}; Theorem 7.7]), obtaining
\begin{align}\label{bdZ1}
Z_\text{bad}&\ll\sum_l\sum_{u\in\mathcal{U}_l^\text{bad}}|\nu_l(u)|\bigg|\sum_{a}\sum_{0<|h|\leq H}\alpha_a\, e\Big(\frac{\pm ahu}{q}\Big)\widehat{W_2}\Big(\frac{hN}{q}\Big)\bigg|\nonumber\\
&\ll\sum_l\bigg(\sum_{u\in\mathcal{U}_l^\text{bad}}|\nu_l(u)|^2\bigg)^{1/2}\bigg(\sum_{u\in\mathcal{U}_l^\text{bad}}\bigg|\sum_{a}\sum_{0<|h|\leq H}\alpha_a\, e\Big(\frac{\pm ahu}{q}\Big)\widehat{W_2}\Big(\frac{hN}{q}\Big)\bigg|^2\bigg)^{1/2}\nonumber\\
&\ll_\eps q^\eps\sum_l \Big(\frac Ml+\frac{XBM}{l^2q^{1/2}}\Big)\Big(\big(q+AH\big)AH\Big)^{1/2}\ll_\eps q^{1+\eps} \big(M+q^{-1/2}XBM\big)\sqrt{\frac AN}.
\end{align}
For $Z_\text{good}$, by Cauchy's inequality and \eqref{bdU} we get
\begin{align*}
Z_\text{good}&\ll \sum_l\bigg(\sum_{u\in\mathcal{U}_l}|\nu_l(u)|^2\bigg)^{1/2}\bigg(\sum_{u\in\mathcal{U}_l^\text{good}}\bigg|\sum_{a}\sum_{0<|h|\leq H}\alpha_a\, e\Big(\frac{\pm ahu}{q}\Big)\widehat{W_2}\Big(\frac{hN}{q}\Big)\bigg|^2\bigg)^{1/2}\\
&\ll_\eps q^\eps\sum_l \frac{(BM)^{1/2}}{l}\bigg(\sum_{u\in\mathcal{U}_l^\text{good}}\bigg|\sum_{a}\sum_{0<|h|\leq H}\alpha_a\, e\Big(\frac{\pm ahu}{q}\Big)\widehat{W_2}\Big(\frac{hN}{q}\Big)\bigg|^2\bigg)^{1/2}.
\end{align*}
Note that for every $u_1,u_2\in\mathcal{U}_l^\text{good}$ we have $\|u_1/q-u_2/q\|>X/q$, where $\| \xi \|$ denotes the distance from $\xi \in \mathbb{R}$ to the nearest integer. Another application of the large sieve inequality yields
\begin{equation}\label{bdZ2}
Z_\text{good}\ll_\eps q^\eps \sum_l \frac{(BM)^{1/2}}{l}\Big(\big(q/X+AH\big)AH\Big)^{1/2}\ll_\eps q^{1+\eps} \Big((BM)^{1/2}X^{-1/2}+\frac{\sqrt{ABM}}{\sqrt{N}}\Big)\sqrt{\frac AN}.
\end{equation}
Combining \eqref{bdZ1}, \eqref{bdZ2} and choosing $X=(q/BM)^{1/3}$ we obtain
\[
\mathcal{E}\ll_\varepsilon \Big(\frac{q}{BM}\Big)^{-1/6+\eps}+q^\eps\Big(\sqrt{\frac AN}+\sqrt{\frac MB}\Big),
\]
and the result follows.
\end{proof}

For Theorem \ref{thm2} we need the following proposition instead of Proposition \ref{prop1}.

\begin{proposition}\label{prop3}
Assume the conditions of Proposition \ref{prop1} and that $ABMN\gg q^{2-\eps}$, $AM\ll BN$. Assume also that $\alpha=\eta*\lambda$, where the sequences $\eta_{a_1}$ and $\lambda_{a_2}$ are supported on the intervals $[A_1,2A_1]$ and $[A_2,2A_2]$ with $A=A_1A_2$, and satisfy $\eta_a,\lambda_a\ll a^\eps$. Then
\begin{equation*}
\mathcal S= \mathcal M_1^++\mathcal M_1^-+\mathcal E ,
\end{equation*}
where $\mathcal{M}_1^\pm$ are defined as in \eqref{mainM2} and
\begin{align*}
\mathcal{E}&\ll_\eps q^{-1+\eps}A^{1/4}B^{1/2}(AM+BN)(MN)^{-1/2}\\
&\qquad\qquad +q^{-3/4+\eps}A^{1/4}B^{1/2}(AM+BN)^{1/2}(MN)^{-1/4}(A_1+A_2)^{1/4}.
\end{align*}
\end{proposition}
\begin{proof}
With $ABMN\gg q^{2-\eps}$ and $AM\ll BN$, we recall from \eqref{contS} that
\begin{equation*}
\mathcal{S}=\mathcal{M}_1^++\mathcal{M}_1^-+ \mathcal{E}+O_C\big(q^{-C}\big)
\end{equation*}
with
\begin{align}\label{E3}
\mathcal{E}\ll &\frac{1}{\sqrt{ABMN}}\sum_{d\leq R}\int_{(\varepsilon)}\int_{(\varepsilon)}\int_{(\varepsilon)}\int_{(\varepsilon)}\int_{x\asymp dMN/(AM+BN)}\nonumber\\
&\qquad\qquad x^{-\text{Re}(w)}\big|\widehat{f_\pm}(u,v,w,z)\big|\big|Z_d(x)\big|dxdudvdwdz,
\end{align}
where
\begin{equation*}
Z_d(x)=d^{w-z}\sum_{\substack{0<|r|\leq R/d\\0<|g|\leq G/d}}\sum_{(a,b)=d}a^{-u}b^{-(v+w)}r^{-z}\alpha_{a}\beta_{b}\,e\Big(\frac{-qrg\overline{a/d}}{b/d}+gx\Big)W_1\Big(\frac{bx}{dM}\Big).
\end{equation*}
Using the fact that $\alpha=\eta*\lambda$ we can write
\begin{align*}
Z_d(x)&=d^{-(u+v+z)}\sum_{d=d_1d_2}\\
&\qquad\ \sum_{\substack{0<|r|\leq R/d\\0<|g|\leq G/d}}\sum_{\substack{a_1,a_2,b\\(a_1a_2,b)=1\\(a_1,d_2)=1}}(a_1a_2)^{-u}b^{-(v+w)}r^{-z}\eta_{d_1a_1}\lambda_{d_2a_2}\beta_{db}\,e\Big(\frac{-qrg\overline{a_1a_2}}{b}+gx\Big)W_1\Big(\frac{bx}{M}\Big).
\end{align*}
Applying Lemma 6 of [\textbf{\ref{H-BJ}}], which is essentially the same as the result of [\textbf{\ref{FI}}], to the above expression with
\[
U\leftrightarrow \frac Bd,\qquad K\leftrightarrow \frac{RG}{d^2}\asymp\frac{ q^\eps (AM+BN)^2}{d^2qMN},\qquad S\leftrightarrow \frac{A_1}{d_1}\qquad \text{and}\qquad T\leftrightarrow \frac{A_2}{d_2}
\]
leads to
\begin{align*}
Z_d(x)&\ll_\eps q^\eps \sum_{d=d_1d_2}U(KST)^{3/4}\big(K^{1/4}+S^{1/4}+T^{1/4}\big)\\
&\ll_\eps d^{-13/4}q^{-3/4+\eps} \sum_{d=d_1d_2}\frac{ A^{3/4}B(AM+BN)^{3/2}}{(MN)^{3/4}}\Big(\frac{(AM+BN)^{1/2}}{d^{1/2}(qMN)^{1/4}}+\frac{A_1^{1/4}}{d_{1}^{1/4}}+\frac{A_2^{1/4}}{d_{2}^{1/4}}\Big).
\end{align*}
Plugging this into \eqref{E3} we obtain the proposition.
\end{proof}

The following proposition is required for Theorem \ref{thm3}.

\begin{proposition}\label{prop4}
Assume the conditions of Proposition \ref{prop3}. Also assume that $\alpha=\eta*\lambda$, $\beta=\nu*\xi$, where the sequences $\eta_{a_1},\lambda_{a_2},\nu_{b_1}, \xi_{b_2}$ are supported on the intervals $[A_1,2A_1],[A_2,2A_2],[B_1,2B_1]$ and $[B_2,2B_2]$ with $A=A_1A_2$, $B=B_1B_2$
, and satisfy $\eta_a,\lambda_a\ll a^\eps$, $\nu_b,\xi_b\ll b^\eps$. Suppose further that
\[
\eta_{a_1}=\eta\Big(\frac{a_1}{A_1}\Big)\qquad\text{and}\qquad \nu_{b_1}=\nu\Big(\frac{b_1}{B_1}\Big),
\] 
where $\eta$ and $\nu$ are smooth functions supported in $[1,2]$ such that $\eta^{(j)},\nu^{(j)}\ll_j q^{\eps}$ for any fixed $j\geq0$. Then
\begin{equation*}
\mathcal S= \mathcal M_1^++\mathcal M_1^-+\mathcal E ,
\end{equation*}
where $\mathcal{M}_1^\pm$ are defined as in \eqref{mainM2} and
\begin{align*}
\mathcal{E}\ll_\eps&\, q^{\vartheta-1/2+\eps}(AB)^{-3/4}(AM+BN)^{3/2}(MN)^{-3/4}\\
&\qquad\qquad\Big(A_2B_2+\frac{(AM+BN)^2}{qMN}\Big)^{1/2}\Big(A_2B_2+\frac{B}{A_1}\Big)^{1/2}
\end{align*}
with $\vartheta=7/64$.
\end{proposition}
\begin{proof}
We proceed from \eqref{contS},
\begin{equation*}
\mathcal{S}=\mathcal{M}_1^++\mathcal{M}_1^-+ \mathcal{E}+O_C\big(q^{-C}\big)
\end{equation*}
with
\begin{align}\label{E4}
\mathcal{E}\ll &\frac{1}{\sqrt{ABMN}}\sum_{d\leq R}\int_{(\varepsilon)}\int_{(\varepsilon)}\int_{(\varepsilon)}\int_{(\varepsilon)}\int_{x\asymp dMN/(AM+BN)}\nonumber\\
&\qquad\qquad x^{-\text{Re}(w)}\big|\widehat{f_\pm}(u,v,w,z)\big|\big|Z_d(x)\big|dxdudvdwdz,
\end{align}
where
\begin{equation*}
Z_d(x)=d^{w-z}\sum_{\substack{0<|r|\leq R/d\\0<|g|\leq G/d}}\sum_{(a,b)=d}a^{-u}b^{-(v+w)}r^{-z}\alpha_{a}\beta_{b}\,e\Big(\frac{-qrg\overline{a/d}}{b/d}+gx\Big)W_1\Big(\frac{bx}{dM}\Big).
\end{equation*}
Since $\alpha=\eta*\lambda$ and $\beta=\nu*\xi$, we get
\begin{align*}
Z_d(x)&=d^{-(u+v+z)}\sum_{d=d_1d_2=d_3d_4}\sum_{\substack{0<|r|\leq R/d\\0<|g|\leq G/d}}\sum_{\substack{a_1,a_2,b_1,b_2\\(a_1a_2,b_1b_2)=1\\(d_1,a_2)=(d_3,b_2)=1}}(a_1a_2)^{-u}(b_1b_2)^{-(v+w)}r^{-z}\\
&\qquad\qquad \eta\Big(\frac{d_1a_1}{A_1}\Big)\nu\Big(\frac{d_3b_1}{B_1}\Big)\lambda_{d_2a_2}\xi_{d_4b_2}\,e\Big(\frac{-qrg\overline{a_1a_2}}{b_1b_2}+gx\Big)W_1\Big(\frac{b_1b_2x}{M}\Big).
\end{align*}


Note that if we apply Weil's bound on the sum over $a_1$ here, we would obtain a slightly weaker result. Instead we apply the Poisson summation formula over $a_1$ and obtain
\begin{align*}
&\sum_{(a_1,b_1b_2)=1}a_1^{-u}\,\eta\Big(\frac{d_1a_1}{A_1}\Big)e\Big(\frac{-qrg\overline{a_1a_2}}{b_1b_2}\Big)={\sum_{x(\text{mod}\ b_1b_2)}}^{\!\!\!\!\!\!\!\!\!*}\ \ e\Big(\frac{-qrg\overline{xa_2}}{b_1b_2}\Big)\sum_{a_1\equiv x(\text{mod}\ b_1b_2)}a_1^{-u}\,\eta\Big(\frac{d_1a_1}{A_1}\Big)\\
&\qquad\qquad=\sum_{k\in\mathbb{Z}}S(qrg\overline{a_2},k;b_1b_2)\int (b_1b_2y)^{-u}\,\eta\Big(\frac{d_1b_1b_2y}{A_1}\Big)e(ky)dy.
\end{align*}
The integral is over $y\asymp d_2A_1/B$, and as before we may restrict the sum to $|k|\leq K$, where
\[
K=\frac{q^\eps B}{d_2A_1}.
\]
Since $|S(qrg\overline{a_2},0;b_1b_2)|\leq (qrg,b_1b_2)$, the contribution of the term $k=0$ to $Z_d(x)$ is
\begin{align*}
&\ll_\eps \frac{q^\eps A_1}{B}  \sum_{d=d_1d_2=d_3d_4}\sum_{\substack{0<|r|\leq R/d\\0<|g|\leq G/d}}\sum_{\substack{a_2,b_1,b_2\\(a_2,b_1b_2)=1\\(d_1,a_2)=(d_3,b_2)=1}} \Big|\nu\Big(\frac{d_3b_1}{B_1}\Big)\lambda_{d_2a_2}\xi_{d_4b_2}\Big|(rg,b_1b_2)d_2\\
&\ll_\eps d^{-3}q^\eps ARG\ll_\eps d^{-3}q^{-1+\eps}\frac{A(AM+BN)^2}{MN}.
\end{align*}
Hence
\begin{equation}\label{Zd4}
Z_d(x)\ll_\eps q^\eps \sum_{d=d_1d_2=d_3d_4}\int_{y\asymp d_2A_1/B}y^{-\text{Re}(u)}|R(y)|dy+d^{-3}q^{-1+\eps}\frac{A(AM+BN)^2}{MN},
\end{equation}
where
\begin{align}\label{Ry}
R(y)&=\sum_{\substack{0<|r|\leq R/d\\0<|g|\leq G/d}}\sum_{\substack{a_2,b_1,b_2\\(a_2,b_1b_2)=1\\(d_1,a_2)=(d_3,b_2)=1}}\sum_{0< |k|\leq K}S(qrg\overline{a_2},k;b_1b_2)\\
&\qquad \qquad a_2^{-u}(b_1b_2)^{-(u+v+w)}r^{-z}\nu\Big(\frac{d_3b_1}{B_1}\Big)\lambda_{d_2a_2}\xi_{d_4b_2}W_1\Big(\frac{b_1b_2x}{M}\Big)\eta\Big(\frac{d_1b_1b_2y}{A_1}\Big)e(gx+ky).\nonumber
\end{align}

We now use the following result, which is a special case of Lemma 4.1 of Watt [\textbf{\ref{W}}].

\begin{lemma}
Let $1\leq N, R,S,U,W<q$, $Z\in\mathbb{R}$ and let $\alpha_r,\beta_s,\gamma_u$ be three sequences of complex numbers supported on $[R,2R], [S,2S]$ and $[1,U]$, and satisfy $\alpha_r\ll R^\eps,\beta_s\ll S^\eps,\gamma_u\ll U^\eps$. For any $r\in[R,2R]$ and $s\in[S,2S]$, suppose that $F_{r,s}$ is a five times continuously differentiable function supported in $[1,2]$ such that $F_{r,s}^{(j)}\ll_j q^\eps$ for any $0\leq j\leq 5$. Then
\begin{align*}
&\sum_{(r,s)=1}\sum_{\substack{0<u\leq U\\0<w\leq W}}\sum_{n}\alpha_r\beta_s\gamma_u e(Zw)F_{r,s}\Big(\frac{n}{N}\Big)S(qu\overline{r},w;ns)\\
&\qquad\qquad\ll_\eps q^\vartheta\big(1+ZW\big)NS\sqrt{RUW}\big(RS+U\big)^{1/2}\big(RS+W\big)^{1/2}\big(X^{2\vartheta}+X^{-3/2}\big),
\end{align*}
where $\vartheta=7/64$ and
\[
X=\frac{NS\sqrt{R}}{\sqrt{qUW}}.
\]
\end{lemma}

Separating the variables $b_1$, $b_2$ in $W_1$ and $\eta$ in \eqref{Ry} using their Mellin transforms, and then applying the above lemma with
\[
R\leftrightarrow \frac{A_2}{d_2},\qquad S\leftrightarrow \frac{B_2}{d_4},\qquad U\leftrightarrow \frac{RG}{d^2}\asymp\frac{ q^\eps (AM+BN)^2}{d^2qMN},\qquad W\leftrightarrow K\asymp \frac{q^\eps B}{d_2A_1},
\]
\[
 N\leftrightarrow \frac{B_1}{d_3},\qquad Z\leftrightarrow \frac{d_2A_1}{B} \qquad \text{and}\qquad X\leftrightarrow \frac{\sqrt{ABMN}}{AM+BN}\ (\ll 1),
\]
we get
\begin{align*}
R(y)&\ll_\eps q^{\vartheta+\eps}\frac{B}{d}\sqrt{\frac{(AM+BN)^2A_2B}{d^2d_2^2qA_1MN}}\Big(\frac{A_2B_2}{d_2d_4}+\frac{(AM+BN)^2}{d^2qMN}\Big)^{1/2}\\
&\qquad\qquad\Big(\frac{A_2B_2}{d_2d_4}+\frac{B}{d_2A_1}\Big)^{1/2}\frac{(AM+BN)^{3/2}}{(ABMN)^{3/4}}\\
&\ll_\eps d^{-2}q^{\vartheta-1/2+\eps}\frac{B}{d_2A_1}\frac{(AM+BN)^{5/2}}{(AB)^{1/4}(MN)^{5/4}}\\
&\qquad\qquad\Big(\frac{A_2B_2}{d_2d_4}+\frac{(AM+BN)^2}{d^2qMN}\Big)^{1/2}\Big(\frac{A_2B_2}{d_2d_4}+\frac{B}{d_2A_1}\Big)^{1/2}.
\end{align*}
Plugging this into \eqref{Zd4} yields
\begin{align*}
Z_d(x)&\ll_\eps d^{-2}q^{\vartheta-1/2+\eps}\frac{(AM+BN)^{5/2}}{(AB)^{1/4}(MN)^{5/4}}\Big(A_2B_2+\frac{(AM+BN)^2}{d^2qMN}\Big)^{1/2}\Big(A_2B_2+\frac{B}{A_1}\Big)^{1/2}\nonumber\\
&\qquad\qquad+d^{-3}q^{-1+\eps}\frac{A(AM+BN)^2}{MN},
\end{align*}
and hence
\begin{align*}
\mathcal{E}&\ll_\eps q^{\vartheta-1/2+\eps}\frac{(AM+BN)^{3/2}}{(AB)^{3/4}(MN)^{3/4}}\Big(A_2B_2+\frac{(AM+BN)^2}{qMN}\Big)^{1/2}\Big(A_2B_2+\frac{B}{A_1}\Big)^{1/2}.
\end{align*}
This completes the proof of the proposition.
\end{proof}

\section{Proof of Theorem \ref{mthm}}\label{smthm}

We proceed from \eqref{M+} and start the evaluation of the off-diagonal terms, $S_{\alpha,\beta}^{+}$. We assume that the sequences $\alpha_a,\beta_b$ are supported on $[A,2A]$ and $[B,2B]$. We remark that our main term analysis is inspired by the nice papers of Young [\textbf{\ref{Y}}] and Zacharias [\textbf{\ref{Z}}].

\subsection{Partition into dyadic intervals}

We first apply a dyadic partition of unity to the sums over $m$ and $n$. Let $W$ be a smooth non-negative function supported in $[1, 2]$ such that
$$
\sum_{M} W \Big ( \frac{x}{M} \Big ) = 1,
$$
where $M$ runs over a sequence of real numbers with $\#\{M: X^{-1}\leq M\leq X\}\ll \log X$. With this partition of unity, we write
\begin{align*}
S_{\alpha,\beta}^{+}=\sum_{M,N} S_{\alpha,\beta}^{+}(M,N),
\end{align*}
where
\begin{align*}
S_{\alpha,\beta}^{+}(M,N)=\sum_{\substack{am\equiv\pm bn(\text{mod}\ q)\\am\ne bn}}\frac{\alpha_a\beta_b}{\sqrt{ab}m^{1/2+\alpha}n^{1/2+\beta}}\,W\Big ( \frac{m}{M} \Big )W \Big ( \frac{n}{N} \Big )V_{+}\Big(\frac{\pi mn}{q}\Big).
\end{align*}
Due to the rapid decay of $V_+$ in Remark \ref{rmkV}, we may assume that $MN\ll q^{1+\eps}$.

Separating the variables $m$ and $n$ in $V_+$ using \eqref{Vpm}, we have
\begin{equation}\label{mainexS}
S_{\alpha,\beta}^{+}(M,N)=\frac{1}{2\pi i}\int_{(\eps)}X_+(s)\Big(\frac q\pi\Big)^s\sum_{\substack{am\equiv\pm bn(\text{mod}\ q)\\am\ne bn}}\frac{\alpha_a\beta_b}{\sqrt{ab}m^{1/2+\alpha+s}n^{1/2+\beta+s}}W\Big ( \frac{m}{M} \Big )W \Big ( \frac{n}{N} \Big )\frac{ds}{s}.
\end{equation}
We now apply Propositions \ref{prop1} and \ref{prop2} to the above sum: given some small $\delta_0>0$, let
\begin{align*}
&\mathcal{A}_1=\big\{(M,N): ABMN\ll q^{2-2\delta_0}\big\},\\
&\mathcal{A}_{2,<}=\big\{(M,N): ABMN\gg q^{2-2\delta_0},BM\ll q^{1-6\delta_0+\eps},A\ll q^{-2\delta_0}N,B\gg q^{2\delta_0}M\big\},\\
&\mathcal{A}_{2,>}=\big\{(M,N): ABMN\gg q^{2-2\delta_0},AN\ll q^{1-6\delta_0+\eps},A\gg q^{2\delta_0}N,B\ll q^{-2\delta_0}M\big\},\\
&\mathcal{A}_3=\big\{(M,N): MN\ll q^{1+\eps}\big\}\backslash\big(\mathcal{A}_1\cup \mathcal{A}_{2,<}\cup \mathcal{A}_{2,>}\big).
\end{align*}
If $(M,N)\in\mathcal{A}_1$ we apply \eqref{prop1-1} of Proposition \ref{prop1}; if $(M,N)\in\mathcal{A}_{2,<}\cup \mathcal{A}_{2,>}$ we apply Proposition \ref{prop2}; and in the remaining case we apply \eqref{prop1-2} of Proposition \ref{prop1}. Then we obtain
\begin{align*}
S_{\alpha,\beta}^{+}(M,N)&=\mathds{1}_{(M,N)\in\mathcal{A}_3}\big(\mathcal{M}_1^+(M,N)+\mathcal{M}_1^-(M,N)\big)\\
&\qquad\qquad+\mathds{1}_{(M,N)\in\mathcal{A}_{2,<}}\mathcal{M}_{2,<}(M,N)+\mathds{1}_{(M,N)\in \mathcal{A}_{2,>}}\mathcal{M}_{2,>}(M,N)+\mathcal{E}(M,N)
\end{align*}
with obvious meaning. More work on the secondary main terms $\mathcal{M}_1^\pm(M,N)$ and $\mathcal{M}_{2,\gtrless}(M,N)$ is required before we can make any use of them. We shall do that in subsections \ref{secondarymt} and \ref{secondarymt2} after showing in the next subsection that the error term $\mathcal{E}(M,N)$ is acceptable.

\subsection{The error term $\mathcal{E}(M,N)$}\label{errorterm}

Without loss of generality, we may assume that $M\ll N$. From \eqref{prop1-1} of Proposition \ref{prop1} and Proposition \ref{prop2}, the error term is $O_\eps(q^{-\delta_0+\eps})$ if $ABMN\ll q^{2-2\delta_0}$, or if 
\begin{equation}\label{cond2}
BM\ll q^{1-6\delta_0+\eps}, \qquad A\ll q^{-2\delta_0}N\qquad\text{and}\qquad B\gg q^{2\delta_0}M.
\end{equation} 
We now consider the remaining case $(M,N)\in\mathcal{A}_3$. Note that if $N\gg q^{\kappa+6\delta_0}$, then $M\ll_\eps q^{1-\kappa-6\delta_0+\eps}$, and all the conditions in \eqref{cond2} are satisfied. So we may assume that $N\ll q^{\kappa+6\delta_0}$. In that case, provided that $\kappa<6/11$ (which we subsequently assume), the first term in \eqref{bdE1} in Proposition \ref{prop1} dominates the second term, i.e.
\begin{align*}
\mathcal{E}(M,N)&\ll_\eps q^{-17/20+\eps}(AB)^{-3/20}(AM+BN)^{17/10}(A+B)^{1/4}(MN)^{-17/20}\\
&\ll _\eps q^{-17/20+\eps}\big(A/B\big)^{3/20}A^{33/20}+q^{-17/20+\eps}A^{31/20}B^{1/10}\\
&\quad +q^{-17/20+\eps}A^{1/10}B^{31/20}\big(N/M)^{17/20}+q^{-17/20+\eps}\big(B/A\big)^{3/20}B^{33/20}\big(N/M)^{17/20}.
\end{align*}
 
We divide into two cases:\\
Case 1: $BM\gg q^{1-6\delta_0+\eps}$ or $B\ll q^{2\delta_0}M$. Then $M\gg q^{1-\kappa-6\delta_0}$, and so $1\leq N/M\ll q^{2\kappa+12\delta_0-1+\eps}$. Note that as $ABMN\gg q^{2-2\delta_0}$, it follows that $q^{1-2\kappa-2\delta_0-\eps}\ll A/B\ll q^{2\kappa+2\delta_0-1+\eps}$, and hence
\begin{align*}
\mathcal{E}(M,N)&\ll _\eps q^{-17/20+3(2\kappa+2\delta_0-1)/20+33\kappa/20+17(2\kappa+12\delta_0-1)/20+\eps}=q^{-37/20+73\kappa/20+21\delta_0/2+\eps};
\end{align*}
Case 2: $A\gg q^{-2\delta_0}N$. As $M\gg q^{2-2\delta_0}/ABN$ we have
\begin{align*}
\mathcal{E}(M,N)&\ll _\eps q^{-17/20+3(2\kappa+2\delta_0-1)/20+33\kappa/20+\eps}+q^{-51/20+17\delta_0/10+\eps}A^{19/20}B^{12/5}N^{17/10}\\
&\qquad\qquad+q^{-51/20+17\delta_0/10+\eps}A^{7/10}B^{53/20}N^{17/10}\\
&\ll_\eps q^{-1+39\kappa/20+3\delta_0/10+\eps}+q^{-51/20+101\kappa/20+51\delta_0/10+\eps}.
\end{align*}

Summing up we obtain
\[
\mathcal{E}(M,N)\ll_\eps q^{-\delta_0+\eps}+q^{-37/20+73\kappa/20+21\delta_0/2+\eps}+q^{-1+39\kappa/20+3\delta_0/10+\eps}+q^{-51/20+101\kappa/20+51\delta_0/10+\eps},
\]
 provided that $\kappa<6/11$. We hence conclude that $\mathcal{E}(M,N)\ll_\eps q^{-\delta_0+\eps}$ for some $\delta_0>0$ if $\kappa<51/101$.

\subsection{The secondary main terms $\mathcal{M}_1^\pm(M,N)$}\label{secondarymt}

For $(M,N)\in\mathcal{A}_3$, from \eqref{prop1-2} of Proposition \ref{prop1} we have
\begin{align*}
\mathcal{M}_1^+(M,N)&=\sum_{r\ne0}\sum_{\substack{d\\(a,b)=1}}\frac{\alpha_{da}\beta_{db}}{d\sqrt{ab}}\frac{1}{2\pi i}\int_{(\eps)} X_+(s)\Big(\frac q\pi\Big)^s\\
&\qquad \int(bx)^{-(1/2+\alpha+s)}\big(ax-qr/b\big)^{-(1/2+\beta+s)}\,W\Big ( \frac{bx}{M} \Big )W \Big ( \frac{abx-qr}{bN} \Big )dx\frac{ds}{s}
\end{align*}
and
\begin{align*}
\mathcal{M}_1^-(M,N)&=\sum_{r\ne0}\sum_{\substack{d\\(a,b)=1}}\frac{\alpha_{da}\beta_{db}}{d\sqrt{ab}}\frac{1}{2\pi i}\int_{(\eps)} X_+(s)\Big(\frac q\pi\Big)^s\\
&\qquad \int(bx)^{-(1/2+\alpha+s)}\big(qr/b-ax\big)^{-(1/2+\beta+s)}\,W\Big ( \frac{bx}{M} \Big )W \Big ( \frac{qr-abx}{bN} \Big )dx\frac{ds}{s}.
\end{align*}
Writing $W$ in terms of its Mellin transform we get
\begin{align*}
&\mathcal{M}_1^+(M,N)=\frac{1}{(2\pi i)^3}\int_{(\eps)}\int_{(c_2)}\int_{(c_1)} X_+(s)\widetilde{W}(u)\widetilde{W}(v)\Big(\frac q\pi\Big)^sM^uN^v\\
&\qquad\quad \sum_{r\ne0}\sum_{\substack{d\\(a,b)=1}}\frac{\alpha_{da}\beta_{db}}{da^{1+\beta+s+v}b^{1+\alpha+s+u}}\int x^{-(1/2+\alpha+s+u)}\big(x-qr/ab\big)^{-(1/2+\beta+s+v)}\,dxdudv\frac{ds}{s}.
\end{align*}
Note that if $r>0$ then the integral over $x$ is restricted to $x>qr/ab$, and if $r<0$ then it is restricted to $x>0$. For absolute convergence, we also need to impose the conditions
\begin{equation}\label{contour1}\begin{cases}
\text{Re}(\alpha+\beta+2s+u+v)>0,\ \text{Re}(\beta+s+v)<1/2 & \quad\text{if }r>0,\\
\text{Re}(\alpha+\beta+2s+u+v)>0,\ \text{Re}(\alpha+s+u)<1/2 & \quad\text{if }r<0.
\end{cases}\end{equation} Under these assumptions, the $x$-integral is equal to (see, for instance, 17.43.21 and 17.43.22 of [\textbf{\ref{GR}}])
\[
\Big(\frac {q|r|}{ab}\Big)^{-(\alpha+\beta+2s+u+v)}\times\begin{cases}
\frac{\Gamma(\alpha+\beta+2s+u+v)\Gamma(1/2-\beta-s-v)}{\Gamma(1/2+\alpha+s+u)} & \quad\text{if }r>0,\\
\frac{\Gamma(\alpha+\beta+2s+u+v)\Gamma(1/2-\alpha-s-u)}{\Gamma(1/2+\beta+s+v)} & \quad\text{if } r<0.
\end{cases}
\]
 Hence
\begin{align*}
\mathcal{M}_1^+(M,N)&=\sum_{r\geq1}\sum_{\substack{d\\(a,b)=1}}\frac{\alpha_{da}\beta_{db}}{dab}\frac{1}{(2\pi i)^3}\int_{(c_2)}\int_{(c_1)} \widetilde{W}(u)\widetilde{W}(v)M^uN^v\\
&\qquad \int_{(\eps)} X_+(s)H^+(s)\Big(\frac q\pi\Big)^s\Big(\frac {q}{a}\Big)^{-(\alpha+s+u)}\Big(\frac {q}{b}\Big)^{-(\beta+s+v)}r^{-(\alpha+\beta+2s+u+v)}\,\frac{ds}{s}dudv,
\end{align*}
where
\[
H^+(s)=\Gamma(\alpha+\beta+2s+u+v)\bigg(\frac{\Gamma(1/2-\beta-s-v)}{\Gamma(1/2+\alpha+s+u)}+\frac{\Gamma(1/2-\alpha-s-u)}{\Gamma(1/2+\beta+s+v)}\bigg).
\]

A similar formula holds for $\mathcal{M}_1^-(M,N)$ with $H^+(s)$ being replaced by $H^-(s)$, where
\[
H^-(s)=\frac{\Gamma(1/2-\alpha-s-u)\Gamma(1/2-\beta-s-v)}{\Gamma(1-\alpha-\beta-2s-u-v)},
\]
and the imposed conditions for the absolute convergence are
\begin{equation}\label{contour2}
\text{Re}(\alpha+s+u)<1/2, \qquad \text{Re}(\beta+s+v)<1/2.
\end{equation}
To keep the symmetry, we write
\[
\mathcal{M}_1^+(M,N)+\mathcal{M}_1^-(M,N)=\mathcal{P}_1(M,N)+\mathcal{P}_2(M,N),
\]
where, for $j=1,2$,
\begin{align}\label{Nintegral}
\mathcal{P}_j(M,N)&=\sum_{r\geq1}\sum_{\substack{d\\(a,b)=1}}\frac{\alpha_{da}\beta_{db}}{dab}\frac{1}{(2\pi i)^3}\int_{(c_2)}\int_{(c_1)} \widetilde{W}(u)\widetilde{W}(v)M^uN^v\\
&\qquad \int_{(\eps)} X_+(s)H_j(s)\Big(\frac q\pi\Big)^s\Big(\frac {q}{a}\Big)^{-(\alpha+s+u)}\Big(\frac {q}{b}\Big)^{-(\beta+s+v)}r^{-(\alpha+\beta+2s+u+v)}\,\frac{ds}{s}dudv,\nonumber
\end{align}
where
\[
H_1(s)=\Gamma(1/2-\beta-s-v)\bigg(\frac{\Gamma(\alpha+\beta+2s+u+v)}{\Gamma(1/2+\alpha+s+u)}+\frac{\Gamma(1/2-\alpha-s-u)}{2\Gamma(1-\alpha-\beta-2s-u-v)}\bigg)
\]
and
\[
H_2(s)=\Gamma(1/2-\alpha-s-u)\bigg(\frac{\Gamma(\alpha+\beta+2s+u+v)}{\Gamma(1/2+\beta+s+v)}+\frac{\Gamma(1/2-\beta-s-v)}{2\Gamma(1-\alpha-\beta-2s-u-v)}\bigg).
\]

\subsubsection{The $r$-sum}

The aim here is to show that we can somehow replace the sum over $r$ in \eqref{Nintegral} by $\zeta(\alpha+\beta+2s+u+v)$. 

We start with $\mathcal{P}_1(M,N)$. Choose $c_1=0$ and $c_2=\eps$. We move the $s$-contour to the right to $\text{Re}(s)=1/2-\eps/3$, crossing a simple pole at $s=1/2-\beta-v$ from the first gamma factor of $H_1(s)$. In doing so we obtain 
\[
\mathcal{P}_1(M,N)=\mathcal{P}_1'(M,N)-\mathcal{R}_1(M,N),
\]
where $\mathcal{P}_1'(M,N)$ is the integral along the new line and
\begin{align*}
\mathcal{R}_1(M,N)&=-\frac32\sum_{r\geq1}\sum_{\substack{d\\(a,b)=1}}\frac{\alpha_{da}\beta_{db}}{dab}\frac{1}{(2\pi i)^2}\int_{(\eps)}\int_{(0)} X_+(1/2-\beta-v)\widetilde{W}(u)\widetilde{W}(v)M^uN^v\\
&\qquad\qquad \Big(\frac q\pi\Big)^{1/2-\beta-v}\Big(\frac {q}{a}\Big)^{-(1/2+\alpha-\beta+u-v)}\Big(\frac {q}{b}\Big)^{-1/2}r^{-(1+\alpha-\beta+u-v)}\,\frac{dudv}{1/2-\beta-v}.
\end{align*}
As moving the $u$-contour in the above expression to $\text{Re}(u)=2\eps$ encountering no pole, we have
\begin{align}\label{R1}
&\mathcal{R}_1(M,N)=-\frac32\sum_{\substack{d\\(a,b)=1}}\frac{\alpha_{da}\beta_{db}}{dab}\frac{1}{(2\pi i)^2}\int_{(\eps)}\int_{(2\eps)} X_+(1/2-\beta-v)\widetilde{W}(u)\widetilde{W}(v)M^uN^v\nonumber\\
&\qquad\quad \Big(\frac q\pi\Big)^{1/2-\beta-v}\Big(\frac {q}{a}\Big)^{-(1/2+\alpha-\beta+u-v)}\Big(\frac {q}{b}\Big)^{-1/2}\zeta(1+\alpha-\beta+u-v)\,\frac{dudv}{1/2-\beta-v}.
\end{align}

With $\mathcal{P}_1'(M,N)$, the $r$-sum can be written as $\zeta(\alpha+\beta+2s+u+v)$. We then shift the $s$-contour back to $\text{Re}(s)=\eps$, but this time crossing two simple poles at $s=1/2-\beta-v$, again, and $s=1/2-(\alpha+\beta+u+v)/2$, because of the zeta-function. We write
\[
\mathcal{P}_1'(M,N)=\mathcal{P}_1''(M,N)+\mathcal{R}_1'(M,N)+\mathcal{R}_1''(M,N)
\]
accordingly, where $\mathcal{R}_1'(M,N)$ and $\mathcal{R}_1''(M,N)$ are the residues at $s=1/2-\beta-v$ and $s=1/2-(\alpha+\beta+u+v)/2$, respectively. It follows that
\[
\mathcal{P}_1(M,N)=\mathcal{P}_1''(M,N)-\big(\mathcal{R}_1(M,N)-\mathcal{R}_1'(M,N)\big)+\mathcal{R}_1''(M,N).
\]

Note that $\mathcal{R}_1'(M,N)$ is the same as $\mathcal{R}_1(M,N)$ in \eqref{R1} but with the $u$-contour being along $\text{Re}(u)=0$. So the difference $\mathcal{R}_1(M,N)-\mathcal{R}_1'(M,N)$ is the residue at $u=v-\alpha+\beta$ of \eqref{R1}. That is to say
\begin{align*}
&\mathcal{R}_1(M,N)-\mathcal{R}_1'(M,N)=-\frac{3}{2q}\sum_{a,b}\frac{\alpha_{a}\beta_{b}}{\sqrt{ab}}\\
&\qquad\quad\frac{1}{2\pi i}\int_{(\eps)} X_+(1/2-\beta-v)\widetilde{W}(v-\alpha+\beta)\widetilde{W}(v) \Big(\frac q\pi\Big)^{1/2-\beta-v}M^{v-\alpha+\beta}N^v\frac{dv}{1/2-\beta-v}.
\end{align*}
We shall make use of the fact that
\begin{align}\label{tildeW}
\widetilde{W}(v)N^v&=N^v\int_{0}^{\infty}x^{v-1}W(x)dx\nonumber\\
&=\sum_{n}n^{v-1}W\Big(\frac nN\Big)+N^v\sum_{n}\int_{n/N}^{(n+1)/N}\bigg(x^{v-1}W(x)-\Big(\frac nN\Big)^{v-1}W\Big(\frac nN\Big)\bigg)dx\nonumber\\
&=\sum_{n}n^{v-1}W\Big(\frac nN\Big)+O_\eps\big(q^\eps N^{\text{Re}(v)-1}|v|\big).
\end{align}
Using this and \eqref{Vpm} we obtain 
\begin{align*}
&\mathcal{R}_1(M,N)-\mathcal{R}_1'(M,N)\\
&\qquad\quad=-\frac{3}{2q}\sum_{\substack{a,b\\m,n}}\frac{\alpha_{a}\beta_{b}}{\sqrt{ab}m^{1/2+\alpha}n^{1/2+\beta}}W\Big(\frac mM\Big)W\Big(\frac nN\Big)V_+\Big(\frac{\pi mn}{q}\Big)+O_\eps\big(q^{\kappa-1/2+\eps} (MN)^{-1}\big).
\end{align*}

For $\mathcal{R}_1''(M,N)$, we have
\[
H_1\Big(\frac{1-\alpha-\beta-u-v}{2}\Big)=\frac{2}{\alpha-\beta+u-v}.
\]
Hence
\begin{align}\label{R''}
&\mathcal{R}_1''(M,N)=\frac{2}{q}\sum_{\substack{d\\(a,b)=1}}\frac{\alpha_{da}\beta_{db}}{d\sqrt{ab}}\frac{1}{(2\pi i)^2}\int_{(\eps)}\int_{(0)} X_+\Big(\frac{1-\alpha-\beta-u-v}{2}\Big) \widetilde{W}(u)\widetilde{W}(v)M^uN^v\nonumber\\
&\qquad\qquad \Big(\frac q\pi\Big)^{(1-\alpha-\beta-u-v)/2}\Big(\frac {a}{b}\Big)^{(\alpha-\beta+u-v)/2}\frac{dudv}{(1-\alpha-\beta-u-v)(\alpha-\beta+u-v)}.
\end{align}

We apply the same argument to $\mathcal{P}_2(M,N)$, but this time choose $c_1=\eps$ and $c_2=0$. Similarly we obtain
\[
\mathcal{P}_2(M,N)=\mathcal{P}_2''(M,N)-\big(\mathcal{R}_2(M,N)-\mathcal{R}_2'(M,N)\big)+\mathcal{R}_2''(M,N).
\]
In particular, note that
\[
H_2\Big(\frac{1-\alpha-\beta-u-v}{2}\Big)=-\frac{2}{\alpha-\beta+u-v}.
\]
So
\begin{align*}
&\mathcal{R}_2''(M,N)=-\frac{2}{q}\sum_{\substack{d\\(a,b)=1}}\frac{\alpha_{da}\beta_{db}}{d\sqrt{ab}}\frac{1}{(2\pi i)^2}\int_{(0)}\int_{(\eps)} X_+\Big(\frac{1-\alpha-\beta-u-v}{2}\Big) \widetilde{W}(u)\widetilde{W}(v)M^uN^v\\
&\qquad\qquad \Big(\frac q\pi\Big)^{(1-\alpha-\beta-u-v)/2}\Big(\frac {a}{b}\Big)^{(\alpha-\beta+u-v)/2}\frac{dudv}{(1-\alpha-\beta-u-v)(\alpha-\beta+u-v)}.\nonumber
\end{align*}
Apart from the minus sign and the changes of the contours, this is the same as $\mathcal{R}_1''(M,N)$. Hence, using \eqref{tildeW} as before,
\begin{align*}
&\mathcal{R}_1''(M,N)+\mathcal{R}_2''(M,N)=-\text{Res}_{u=v-\alpha+\beta}\ \text{of \eqref{R''}}\\
&\quad\qquad=-\frac{1}{q}\sum_{a,b}\frac{\alpha_{a}\beta_{b}}{\sqrt{ab}}\frac{1}{2\pi i}\int_{(\eps)} X_+(1/2-\beta-v) \Big(\frac q\pi\Big)^{1/2-\beta-v}\\
&\quad\quad\qquad\qquad\qquad\widetilde{W}(1/2-\beta-v)\widetilde{W}(v)M^{1/2-\beta-v}N^v\frac{dv}{(1/2-\beta-v)}\\
&\quad\qquad=-\frac{1}{q}\sum_{\substack{a,b\\m,n}}\frac{\alpha_{a}\beta_{b}}{\sqrt{ab}m^{1/2+\alpha}n^{1/2+\beta}}W\Big(\frac mM\Big)W\Big(\frac nN\Big)V_+\Big(\frac{\pi mn}{q}\Big)+O_\eps\big(q^{\kappa-1/2+\eps} M^{-1/2}N^{-1}\big).
\end{align*}

We sum up the calculations in this subsection in the following proposition.

\begin{proposition}\label{propM1}
For $(M,N)\in\mathcal{A}_3$ we have
\begin{align*}
&\mathcal{M}_1^+(M,N)+\mathcal{M}_1^-(M,N)=\mathcal{P}_1''(M,N)+\mathcal{P}_2''(M,N)\\
&\qquad\qquad+\frac{2}{q}\sum_{\substack{a,b\\m,n}}\frac{\alpha_{a}\beta_{b}}{\sqrt{ab}m^{1/2+\alpha}n^{1/2+\beta}}W\Big(\frac mM\Big)W\Big(\frac nN\Big)V_+\Big(\frac{\pi mn}{q}\Big)+O_\eps\big(q^{7\kappa/2+5\delta_0-5/2+\eps}\big),
\end{align*}
where, for $j=1,2$,
\begin{align}\label{N''}
&\mathcal{P}_j''(M,N)=\sum_{\substack{d\\(a,b)=1}}\frac{\alpha_{da}\beta_{db}}{dab}\frac{1}{(2\pi i)^3}\int_{(c_2)}\int_{(c_1)} \widetilde{W}(u)\widetilde{W}(v)M^uN^v\\
&\qquad\qquad \int_{(\eps)} X_+(s)H_j(s)\Big(\frac q\pi\Big)^s\Big(\frac {q}{a}\Big)^{-(\alpha+s+u)}\Big(\frac {q}{b}\Big)^{-(\beta+s+v)}\zeta(\alpha+\beta+2s+u+v)\,\frac{ds}{s}dudv.\nonumber
\end{align}
\end{proposition}

\subsection{The secondary main terms $\mathcal{M}_{2,\gtrless}(M,N)$}\label{secondarymt2}

Applying Proposition \ref{prop2} to $S_{\alpha,\beta}^+(M,N)$ for the pairs $(M,N)\in\mathcal{A}_{2,<}$ we get
\begin{align*}
\mathcal{M}_{2,<}(M,N)=&\,\frac 2q\sum_{a,b}\frac{\alpha_{a}\beta_{b}}{\sqrt{ab}}\frac{1}{2\pi i}\int_{(\eps)} X_+(s)\Big(\frac q\pi\Big)^s\\
&\qquad\qquad \sum_m m^{-(1/2+\alpha+s)}\,W\Big ( \frac{m}{M} \Big )\int x^{-(1/2+\beta+s)}\,W\Big ( \frac{x}{N} \Big )dx\frac{ds}{s}.
\end{align*}
Recall from \eqref{tildeW} that
\begin{align*}
\int_{0}^{\infty}x^{u-1}W\Big(\frac xN\Big)dx=\sum_{n}n^{u-1}W\Big(\frac nN\Big)+O_\eps\big(q^\eps N^{\text{Re}(u)-1}|u|\big).
\end{align*}
The contribution of the $O$-term to $\mathcal{M}_{2,<}(M,N)$ is $O_\eps\big(q^{-1+\eps}\sqrt{ABM/N}\big)$. Hence
\begin{equation}\label{propM2}
\mathcal{M}_{2,<}(M,N)=\frac{2}{q}\sum_{\substack{a,b\\m,n}}\frac{\alpha_{a}\beta_{b}}{\sqrt{ab}m^{1/2+\alpha}n^{1/2+\beta}}W\Big(\frac mM\Big)W\Big(\frac nN\Big)V_+\Big(\frac{\pi mn}{q}\Big)+O_\eps\big(q^{-1/2-4\delta_0+\eps}\big).
\end{equation}
The same expression holds for $\mathcal{M}_{2,>}(M,N)$.

\subsection{Assembling the partition of unity}

Recall that Proposition \ref{propM1} holds when $(M,N)\in\mathcal{A}_3$ and \eqref{propM2} holds when $(M,N)\in\mathcal{A}_{2,\gtrless}$. Summing up we obtain
\begin{align*}
&S_{\alpha,\beta}^{+}=\sum_{(M,N)\in\mathcal{A}_3} \big(\mathcal{P}_1''(M,N)+\mathcal{P}_2''(M,N)\big)\\
&\qquad\qquad+\sum_{(M,N)\notin\mathcal{A}_{1}}\frac{2}{q}\sum_{\substack{a,b\\m,n}}\frac{\alpha_{a}\beta_{b}}{\sqrt{ab}m^{1/2+\alpha}n^{1/2+\beta}}W\Big(\frac mM\Big)W\Big(\frac nN\Big)V_+\Big(\frac{\pi mn}{q}\Big)+O_\eps\big(q^{-\delta_0+\eps}\big)
\end{align*} 
for some $\delta_0>0$.

The condition $(M,N)\notin\mathcal{A}_{1}$ in the second sum may be removed at the cost of an error of size $O_\eps\big(q^{-\delta_0+\eps}\big)$. This allows us to extend the summation over all $(M,N)$, and thus to remove the partition of unity. For the first sum, the following result shall allow us to add all the missing pairs $(M,N)$.

\begin{lemma}
With $\mathcal{P}_j''(M,N)$, $j=1,2$, defined as in \eqref{N''} we have
\[
\mathcal{P}_j''(M,N)\ll_\eps q^\eps\min\bigg\{\frac{\sqrt{ABMN}}{q},\sqrt{\frac AN},\sqrt{\frac BM}\bigg\}.
\]
\end{lemma}
\begin{proof}
Recall from \eqref{contour1} and \eqref{contour2} that for $j=1,2$, $H_j$'s have rapid decay as any of the variables gets large in the imaginary directions. So we have, trivially, 
\[
\mathcal{P}_j''(M,N)\ll_\eps q^{-(\sigma+c_1+c_2)+\eps}M^{c_1}N^{c_2}A^{\sigma+c_1}B^{\sigma+c_2},
\]
provided that
\[
0<2\sigma+c_1+c_2<1,\qquad \sigma+c_1<\frac12\qquad\text{and}\qquad \sigma+c_2<\frac12.
\]
The lemma follows by choosing various suitable values of $\sigma$, $c_1$ and $c_2$.
\end{proof}

The above lemma implies that for $j=1,2$,
\begin{eqnarray*}
\sum_{(M,N)\in\mathcal{A}_3}\mathcal{P}_j''(M,N)=\sum_{M,N}\mathcal{P}_j''(M,N)+O_\eps\big(q^{-\delta_0+\eps}\big).
\end{eqnarray*}
Now we can apply the following result of Young [\textbf{\ref{Y}}; p. 30].

\begin{lemma}\label{sumunity}
Let $F(s_1,s_2)$ be an entire function of rapid decay in each variable in a fixed strip $|\emph{Re}(s_j)|\leq C$, $j=1,2$. Then we have
\[
\sum_{M,N}\frac{1}{(2\pi i)^2}\int_{(c_2)}\int_{(c_1)}F(s_1,s_2)\widetilde{W}(s_1)\widetilde{W}(s_2)ds_1ds_2=F(0,0).
\]
\end{lemma}

In view of this we get
\begin{align*}
&\sum_{(M,N)\in\mathcal{A}_3}\big(\mathcal{P}_1''(M,N)+\mathcal{P}_2''(M,N)\big)=\mathcal{N}_{\alpha,\beta}^++O_\eps\big(q^{-\delta_0+\eps}\big),
\end{align*}
where
\[
\mathcal{N}_{\alpha,\beta}^+=\sum_{\substack{d\\(a,b)=1}}\frac{\alpha_{da}\beta_{db}}{dab}\frac{1}{2\pi i}\int_{(\eps)} X_+(s)H(s)\Big(\frac q\pi\Big)^s\Big(\frac {q}{a}\Big)^{-(\alpha+s)}\Big(\frac {q}{b}\Big)^{-(\beta+s)}\zeta(\alpha+\beta+2s)\,\frac{ds}{s}
\]
and
\begin{align*}
H(s)=&\frac{\Gamma(\alpha+\beta+2s)\Gamma(1/2-\beta-s)}{\Gamma(1/2+\alpha+s)}+\frac{\Gamma(\alpha+\beta+2s)\Gamma(1/2-\alpha-s)}{\Gamma(1/2+\beta+s)}\\
&\qquad\qquad+\frac{\Gamma(1/2-\alpha-s)\Gamma(1/2-\beta-s)}{\Gamma(1-\alpha-\beta-2s)}.
\end{align*}
Using Lemma 8.2 of [\textbf{\ref{Y}}] we have
\[
H(s)=\pi^{1/2}\frac{\Gamma(\frac{\alpha+\beta+2s}{2})\Gamma(\frac{1/2-\alpha-s}{2})\Gamma(\frac{1/2-\beta-s}{2})}{\Gamma(\frac{1-\alpha-\beta-2s}{2})\Gamma(\frac{1/2+\alpha+s}{2})\Gamma(\frac{1/2+\beta+s}{2})}.
\]
So
\begin{align*}
\mathcal{N}_{\alpha,\beta}^+&=\frac{\pi^{1/2}q^{-(\alpha+\beta)}}{\Gamma(\frac{1/2+\alpha}{2})\Gamma(\frac{1/2+\beta}{2})}\sum_{\substack{d\\(a,b)=1}}\frac{\alpha_{da}\beta_{db}}{da^{1-\alpha}b^{1-\beta}}\\
&\qquad\qquad\frac{1}{2\pi i}\int_{(\eps)} G(s)\frac{\Gamma(\frac{\alpha+\beta+2s}{2})\Gamma(\frac{1/2-\alpha-s}{2})\Gamma(\frac{1/2-\beta-s}{2})}{\Gamma(\frac{1-\alpha-\beta-2s}{2})}\Big(\frac{ab} {\pi q}\Big)^{s}\zeta(\alpha+\beta+2s)\,\frac{ds}{s}.
\end{align*}
We apply the functional equation,
\begin{align*}
&\pi^{-(\alpha+\beta+2s)/2}\Gamma\Big(\frac{\alpha+\beta+2s}{2}\Big)\zeta(\alpha+\beta+2s)\\
&\qquad\qquad=\pi^{-(1-\alpha-\beta-2s)/2}\Gamma\Big(\frac{1-\alpha-\beta-2s}{2}\Big)\zeta(1-\alpha-\beta-2s),
\end{align*}
and change the variable $s\rightarrow -s$ to obtain
\begin{align}\label{N+}
\mathcal{N}_{\alpha,\beta}^+=-\Big(\frac{q}{\pi}\Big)^{-(\alpha+\beta)}\sum_{\substack{d\\(a,b)=1}}\frac{\alpha_{da}\beta_{db}}{da^{1-\alpha}b^{1-\beta}}\frac{1}{2\pi i}\int_{(-\eps)} X_-(s)\Big(\frac {\pi ab}{ q}\Big)^{-s}\zeta(1-\alpha-\beta+2s)\,\frac{ds}{s}.
\end{align}

We conclude that

\begin{proposition}\label{S+}
We have
\begin{align*}
S_{\alpha,\beta}^{+}&=\mathcal{N}_{\alpha,\beta}^++\frac{2}{q}\sum_{\substack{a,b\\m,n}}\frac{\alpha_{a}\beta_{b}}{\sqrt{ab}m^{1/2+\alpha}n^{1/2+\beta}}V_+\Big(\frac{\pi mn}{q}\Big)+O_\eps\big(q^{-\delta_0+\eps}\big),
\end{align*}
where $\mathcal{N}_{\alpha,\beta}^+$ is defined as in \eqref{N+}.
\end{proposition}

\subsection{Combining the main terms and secondary main terms}

Section \ref{initial} and Proposition \ref{S+} imply that
\[
I_{\alpha,\beta}=\Big(\mathcal{M}_{\alpha,\beta}^++\mathcal{N}_{\alpha,\beta}^+\Big)+\Big(\frac{q}{\pi}\Big)^{-(\alpha+\beta)}\Big(\mathcal{M}_{-\beta,-\alpha}^-+\mathcal{N}_{-\beta,-\alpha}^-\Big)+O_\eps\big(q^{-\delta_0+\eps}\big),
\]
where $\mathcal{M}^\pm$, $\mathcal{N}^\pm$ are defined as in \eqref{M+} and \eqref{N+}. We have, for instance,
\begin{align*}
\Big(\frac{q}{\pi}\Big)^{-(\alpha+\beta)}\mathcal{M}_{-\beta,-\alpha}^-+\mathcal{N}_{\alpha,\beta}^+&=\Big(\frac{q}{\pi}\Big)^{-(\alpha+\beta)}\sum_{\substack{da,db\leq q^\kappa\\(a,b)=1}}\frac{\alpha_{da}\beta_{db}}{da^{1-\alpha}b^{1-\beta}}\\
&\qquad\frac{1}{2\pi i}\bigg(\int_{(\eps)}-\int_{(-\eps)}\bigg) X_-(s)\Big(\frac {\pi ab}{ q}\Big)^{-s}\zeta(1-\alpha-\beta+2s)\,\frac{ds}{s}\\
&=\text{Res}_{s=0}\\
&=X_-(0)\Big(\frac{q}{\pi}\Big)^{-(\alpha+\beta)}\zeta(1-\alpha-\beta)\sum_{\substack{da,db\leq q^\kappa\\(a,b)=1}}\frac{\alpha_{da}\beta_{db}}{da^{1-\alpha}b^{1-\beta}}.
\end{align*}
Note that the pole at $s=(\alpha+\beta)/2$ of the zeta-function is cancelled by the function $G$. A similar expression holds for the combination of the other two terms, and Theorem \ref{mthm} follows.

\section{Proofs of Theorems \ref{thm2} and \ref{thm3}}\label{sthm2}

\subsection{Proof of Theorem \ref{thm2}}

We argue the same as in the proof of Theorem \ref{mthm}. The only difference is that we also apply Proposition \ref{prop3} to \eqref{mainexS}: if $(M,N)\in\mathcal{A}_1$ we apply \eqref{prop1-1} of Proposition \ref{prop1}; if $(M,N)\in\mathcal{A}_{2,<}\cup \mathcal{A}_{2,>}$ we apply Proposition \ref{prop2}; and in the remaining case we apply Proposition \ref{prop3}. So it remains to check that the error term $\mathcal{E}(M,N)$ is acceptable when $(M,N)\in\mathcal{A}_3$.

As in Subsection \ref{errorterm}, if $(M,N)\in\mathcal{A}_3$, then we may assume that $M, N\ll q^{\kappa+6\delta_0}$. Without loss of generality, let us assume that $AM\ll BN$. Then from Proposition \ref{prop3}, 
\begin{align*}
\mathcal{E}(M,N)&\ll_\eps q^{-1+\eps}A^{1/4}B^{3/2}(N/M)^{1/2}+q^{-3/4+\eps}A^{1/4}B(N/M)^{1/4}(A_1+A_2)^{1/4}.
\end{align*}
Since $M\gg q^{2-2\delta_0}/ABN$, it follows that $N/M\ll q^{-2+2\delta_0}ABN^2\ll q^{-2+4\kappa+14\delta_0}$, and hence
\begin{align*}
\mathcal{E}(M,N)&\ll_\eps q^{-2+15\kappa/4+7\delta_0+\eps}+q^{-5/4+9\kappa/4+7\delta_0/2+\eps}(A_1+A_2)^{1/4}.
\end{align*}
Thus $\mathcal{E}(M,N)\ll_\eps q^{-\delta_0+\eps}$ for some $\delta_0>0$ if $9\kappa+\max\{\kappa_1,\kappa_2\}<5$.

\subsection{Proof of Theorem \ref{thm3}}

Again we follow the arguments in the proof of Theorem \ref{mthm}. The only difference is that if $(M,N)\in\mathcal{A}_3$, then we apply Proposition \ref{prop4} to \eqref{mainexS}. We only need to verify that the error term $\mathcal{E}(M,N)$ is negligible in this case.

As above, we may assume that $M, N\ll q^{\kappa+6\delta_0}$, and, without loss of generality, that $AM\ll BN$. In view of Proposition \ref{prop4},
\begin{align*}
\mathcal{E}(M,N)\ll_\eps&\, q^{\vartheta-1/2+\eps}(BN/AM)^{3/4}\Big(A_2B_2+\frac{B^2N}{qM}\Big)^{1/2}\Big(A_2B_2+\frac{B}{A_1}\Big)^{1/2}
\end{align*}
with $\vartheta=7/64$. Since $M\gg q^{2-2\delta_0}/ABN$, we get
\begin{align*}
\mathcal{E}(M,N)&\ll_\eps\, q^{\vartheta-2+3\kappa/2+21\delta_0/2+\eps}B^{3/2}\big(A_2B_2+q^{-3+2\kappa+14\delta_0}AB^3\big)^{1/2}\Big(A_2B_2+\frac{B}{A_1}\Big)^{1/2}.
\end{align*}
As $MN\ll q^{1+\eps}$, we have $AB\gg q^{1-2\delta_0-\eps}$, and so $B/A_1\ll q^{-1+2\delta_0+\eps}A_2B^2$. Hence
\begin{align*}
\mathcal{E}(M,N)&\ll_\eps\, q^{\vartheta-2+3\kappa+2\kappa_2+21\delta_0/2+\eps}+q^{\vartheta-5/2+4\kappa+3\kappa_2/2+23\delta_0/2+\eps}+q^{\vartheta-7/2+6\kappa+\kappa_2+35\delta_0/2+\eps}\\
&\ll_\eps\, q^{\vartheta-1/2+7\kappa_2+35\delta_0/2+\eps}.
\end{align*}
Thus $\mathcal{E}(M,N)\ll_\eps q^{-\delta_0+\eps}$ for some $\delta_0>0$ if $\kappa_2<1/14-\vartheta/7$.

\section{Proof of Theorem \ref{thm: third moment}}

The lower bound for \eqref{eq: third moment} is relatively straightforward, and follows from the work of Rudnick and Soundararajan [\textbf{\ref{RS}}]. We therefore focus on the upper bound. Our approach utilizes a combination of ideas from Heath-Brown [\textbf{\ref{hb}}] and Bettin, Chandee and Radziwi\l\l\ [\textbf{\ref{BCR}}], as well as our Theorem \ref{mthm} on the twisted second moment of Dirichlet $L$-functions. Heath-Brown [\textbf{\ref{hb}}; Theorem 1] previously obtained Theorem \ref{thm: third moment} assuming the Generalized Riemann Hypothesis, and Bettin, Chandee and Radziwi\l\l \ [\textbf{\ref{BCR}}; Corollary 2] obtained the analogue of Theorem \ref{thm: third moment} for the Riemann zeta-function $\zeta(s)$.

Let us define
\begin{align*}
M(q) = \sideset{}{^*}\sum_{\chi(\text{mod}\ q)}   \left|L \left( \tfrac{1}{2},\chi\right) \right|^3,
\end{align*}
so that the upper bound in \eqref{eq: third moment} follows from the estimate
\begin{align*}
M(q) &\ll q (\log q)^{9/4}.
\end{align*}
We follow Heath-Brown (see [\textbf{\ref{hb}}; p. 408--409]) and first obtain an upper bound for $M(q)$ in terms of an integral.

As $L(s,\chi)$ is an analytic function, we have
\begin{align}\label{eq: L subharmonic inequality}
\left|L \left( \tfrac{1}{2},\chi\right) \right|^3 &\leq \frac{1}{2\pi} \int_0^{2\pi } \left|L \left( \tfrac{1}{2}+re^{i\theta},\chi\right) \right|^3 d\theta
\end{align}
for any $r \geq 0$. 
We multiply both sides of \eqref{eq: L subharmonic inequality} by $r$ and integrate from $0$ to $R$, obtaining
\begin{align}\label{eq: upper bound for central value in terms of integral}
\left|L \left( \tfrac{1}{2},\chi\right) \right|^3 &\leq \frac{1}{\text{meas}(D)} \int_D \left|L \left( \tfrac{1}{2}+z,\chi\right) \right|^3 dA,
\end{align}
where $D = \{z : |z| \leq R\}$ and $dA$ denotes the area measure. We choose $R = (\log q)^{-1}$, and then the real part of $1/2+z$ satisfies
\begin{align*}
\frac{1}{2} - \frac{1}{\log q} \leq \text{Re}\Big( \frac{1}{2}+z \Big) \leq \frac{1}{2} + \frac{1}{\log q}.
\end{align*}
Now define a function
\begin{align*}
W_\rho (s)=W(s) = \frac{q^{\rho(s - 1/2)}-1}{(s -1/2)\log q},
\end{align*}
where $\rho > 0$ is a parameter at our disposal (we eventually take $\rho$ to be rather small). For $z \in D$ we have
\begin{align}\label{eq: lower bound for W on D}
\left|W \left( \tfrac{1}{2}+z \right) \right| \geq \frac{\rho}{2},
\end{align}
provided that $\rho \leq 1/2$, say. By positivity, we obtain from \eqref{eq: upper bound for central value in terms of integral} and \eqref{eq: lower bound for W on D} that
\begin{align*}
\left|L \left( \tfrac{1}{2},\chi\right) \right|^3 &\ll_\rho (\log q)^2\int_D \left|W \left( \tfrac{1}{2}+z \right) \right|^6 \left|L \left( \tfrac{1}{2}+z,\chi\right) \right|^3 dA \\
&\leq (\log q)^2 \int_{- 1/\log q}^{1/\log q} \int_{-\infty}^\infty \left|W \left( \tfrac{1}{2}+\gamma + it \right) \right|^6 \left|L \left( \tfrac{1}{2}+\gamma + it,\chi\right) \right|^3 dt \ d\gamma.
\end{align*}

We define
\begin{align*}
J(\gamma) = \sideset{}{^*}\sum_{\chi(\text{mod}\ q)}  \int_{-\infty}^\infty \left|W \left( \tfrac{1}{2}+\gamma + it \right) \right|^6 \left|L \left( \tfrac{1}{2}+\gamma + it,\chi\right) \right|^3 dt.
\end{align*}
We have therefore obtained
\begin{align*}
M(q) &\ll_\rho (\log q)^2 \int_{- 1/\log q}^{1/\log q} J(\gamma) \ d \gamma.
\end{align*}
By [\textbf{\ref{hb}}; Lemma 4] we have
\begin{align*}
J(\gamma) \ll_A\, q(\log q)^{-1} + J\Big(\frac{A}{\log q} \Big),
\end{align*}
where $A \geq 1$ is a parameter at our disposal. Thus,
\begin{align}\label{eq: upper bound for M in terms of J}
M(q) &\ll_{\rho,A} \, q + (\log q)J\Big(\frac{A}{\log q} \Big).
\end{align}
We shall eventually take $A$ to be a large, but fixed, constant.

In order to apply our main theorem, we need to truncate the integral in the definition of $J(\gamma)$ to $|t| \leq T$, for some relatively small $T$. This can be done easily because of the decay of $W$ in vertical strips. For $|t| \geq 1$ we have
\begin{align*}
\left|W \Big( \frac{1}{2}+\frac{A}{\log q} + it \Big) \right|^6 \ll_{\rho,A} \frac{1}{t^6},
\end{align*}
so that
\begin{align*}
&\sideset{}{^*}\sum_{\chi(\text{mod}\ q)} \int_{|t| \geq T} \left|W \Big( \frac{1}{2}+\frac{A}{\log q} + it \Big) \right|^6 \left|L \Big( \frac{1}{2}+\frac{A}{\log q} + it,\chi\Big) \right|^3 dt \\
&\qquad\qquad\ll_{\rho,A} \ \sideset{}{^*}\sum_{\chi(\text{mod}\ q)} \int_{|t| \geq T}   \left|L \Big( \frac{1}{2}+\frac{A}{\log q} + it,\chi\Big) \right|^3\frac{dt}{t^6} .
\end{align*}
We break the integral into dyadic segments, so we must estimate
\begin{align*}
&\sideset{}{^*}\sum_{\chi(\text{mod}\ q)} \int_{U \leq |t| \leq 2U} \left|L \Big( \frac{1}{2}+\frac{A}{\log q} + it,\chi\Big) \right|^3\frac{dt}{t^6}\\
&\qquad\qquad \ll \frac{1}{U^6} \sideset{}{^*}\sum_{\chi(\text{mod}\ q)}  \int_{-2U}^{2U}\left|L \Big( \frac{1}{2}+\frac{A}{\log q} + it,\chi\Big) \right|^3 dt
\end{align*}
for $U \geq T$. By H\"older's inequality and Theorem 10.1 of [\textbf{\ref{mont}}], we find that this latter quantity is
\begin{align*}
\ll \frac{(qU)^{1/4}}{U^6} \left( \ \sideset{}{^*}\sum_{\chi(\text{mod}\ q)}  \int_{-2U}^{2U} \left|L \Big( \frac{1}{2}+\frac{A}{\log q} + it,\chi\Big) \right|^4 dt \right)^{3/4}\ll \frac{q (\log qU)^3}{ U^{5}}.
\end{align*}
Summing over $U=2^j$ with $U \geq T=\log q$, we find that
\begin{align}\label{eq: J and J1}
J\Big(\frac{A}{\log q}\Big) &= \ J_1\Big(  \frac{A}{\log q}\Big) + O_{\rho,A}\big(q(\log q)^{-1}\big),
\end{align}
where
\begin{align}\label{fmlJ1}
&J_1(\gamma) = \sideset{}{^*}\sum_{\chi(\text{mod}\ q)}  \int_{-\log q}^{\log q} \left|W \left(\tfrac12+ \gamma + it \right) \right|^6\left|L \left( \tfrac12+\gamma + it,\chi\right) \right|^3dt.
\end{align}

We next require an upper bound for $L(1/2+A/\log q+it,\chi)$. Here we mostly follow the arguments of [\textbf{\ref{BCR}}; Section 6]. 
\begin{lemma}\label{lem: BCR lemma 2 analog}
Let $G$ be a compactly supported function. Suppose that $F(u) = -G'(u)$ for $u > 0$ and $F$ is three times continuously differentiable and compactly supported. 
Then
\begin{align*}
\sum_n \frac{\chi(n)}{n^s} G \Big(\frac{\log n}{\log x} \Big) = \frac{1}{2\pi i}\int_{(c)} L (s+w,\chi) \widehat{F} \Big(\frac{-iw \log x}{2\pi} \Big) \frac{dw}{w},
\end{align*}
where $c > \max\{1-\emph{Re}(s),0\}$, $x > 1$ and $\widehat{F}$ denotes the Fourier transform of $F$.
\end{lemma}
\begin{proof}
Argue as in [\textbf{\ref{BCR}}; Lemma 2].
\end{proof}

Let $\delta >0$ be a small positive parameter to be chosen later. We introduce another parameter $\theta$ which satisfies $\delta < \theta < 1$. We define
\begin{align*}
\widehat{F}(z) &= e^{2\pi i (\theta - \delta)z}\bigg(\frac{e^{2\pi i (1-\theta)z}-1}{2\pi i (1-\theta)z}\bigg)^N,
\end{align*}
where $N \geq 10$ is a bounded integer (actually, $N=10$ suffices). Then $F$ is compactly supported on $[\theta - \delta, \theta - \delta + (1-\theta)N]$. For $u > 0$, we define
\begin{align*}
G(u) &= 1 - \int_0^u F(v) dv,
\end{align*}
and $G(u) = 0$ for $u \leq -1$. We let $G$ decay smoothly to zero on the interval $[-1,0]$. Thus $F(u) = -G'(u)$ for all $u >0$. We have $G(u) = 1$ for $0 < u < \theta - \delta$, and $G(u) = 0$ for $u > \theta - \delta + (1-\theta) N$. Lastly, $G$ is $N$ times differentiable, and therefore $\widehat{G}(u) \ll (1+|u|)^{-N}$.

We now let $x \leq q^{1/2 + 1/300}$, and choose our parameters $\delta$, $\theta$ so that $\delta = 2(N-1)(1 - \theta)$, and $\theta = (\log y)/(\log x)$ with $y = q^{1/2+2\delta}$. Note that we want $\delta$ to be small enough so that $x \leq q^{1/2 + 1/300}$. We remark that we shall eventually choose our parameter $\rho$ to be sufficiently small in terms of $\delta$.

With such choice of $\delta$ and $\theta$ we have
\begin{align*}
\widehat{F} \bigg(\frac{-iw \log x}{2\pi} \bigg) &= (yx^{-\delta})^w \bigg(\frac{(x/y)^w-1}{w(1-\theta)(\log x)} \bigg)^N.
\end{align*}
Using Lemma \ref{lem: BCR lemma 2 analog} with $s = \sigma + it$ and $\sigma = 1/2 + A/\log q$, we shift the line of integration to $\text{Re}(w) =1/2- \sigma$, thereby obtaining
\begin{align}\label{eq: upper bound for L}
|L(s,\chi)| &\leq \bigg|\sum_n \frac{\chi(n)}{n^s} G \Big(\frac{\log n}{\log x}\Big) \bigg| + (4N)^N \frac{(yx^{-\delta})^{1/2-\sigma}}{(\delta \log x)^N}\int_{-\infty}^\infty \frac{\left|L \left(\frac{1}{2}+it+iv,\chi \right) \right|}{\big((\sigma -1/2)^2 + v^2\big)^{(N+1)/2}}dv.
\end{align}

Let
\begin{align*}
c(n) &= \sum_{\substack{n=ab\\ a,b \leq x}} d_{1/2}(a) d_{1/2}(b),
\end{align*}
where $d_{1/2}(n)$ are the Dirichlet coefficients of $\zeta(s)^{1/2}$. Then $c(n) = 1$ for $n \leq x$. Since $G (u) = 0$ for $u > 1$, we obtain by Fourier inversion
\begin{align}\label{idG}
\sum_n \frac{\chi(n)}{n^s} G \Big(\frac{\log n}{\log x}\Big) &= \sum_n \frac{c(n)\chi(n)}{n^s} G \Big(\frac{\log n}{\log x}\Big)\nonumber\\
&= \frac{\log x}{2\pi}\int_{-\infty}^{\infty}\bigg(\sum_{a \leq x} \frac{d_{1/2}(a)\chi(a)}{a^{s+iv}} \bigg)^2 \widehat{G}\Big( \frac{v \log x}{2\pi}\Big) dv.
\end{align}
We then multiply both sides of \eqref{eq: upper bound for L} by $|W(s)|^6|L(s,\chi)|^2$, integrate over $t$, sum on $\chi$, apply \eqref{fmlJ1} and \eqref{idG} to obtain
\begin{align}\label{eq: upper bound for J with cal E}
J_1 \Big( \frac{A}{\log q}\Big) \leq \mathcal{M} + \mathcal{E},
\end{align}
where
\begin{align*}
\mathcal{M}=&\,\frac{\log x}{2\pi}\int_{-\log q}^{\log q} \int_{-\infty}^\infty \bigg|\widehat{G}\Big( \frac{v \log x}{2\pi}\Big)\bigg|\bigg|W \Big(\frac{1}{2} + \frac{A}{\log q}+it\Big) \bigg|^6\\
&\qquad\qquad\sideset{}{^*}\sum_{\chi(\text{mod}\ q)}  \bigg|L \Big(\frac{1}{2} + \frac{A}{\log q}+it,\chi\Big) \bigg|^2  \bigg|\sum_{a \leq x} \frac{d_{1/2}(a) \chi(a)}{a^{1/2+A/\log q+i(t+v)}} \bigg|^2dvdt
\end{align*}
and
\begin{align*}
\mathcal{E} =&\, (4N)^N \frac{(yx^{-\delta})^{-A/\log q}}{(\delta \log x)^N}\int_{-\infty}^{\infty} \int_{-\infty}^\infty  \frac{\big| W (1/2 + A/\log q + it ) \big|^6}{\big(A^2/(\log q)^2 + v^2\big)^{(N+1)/2}} \\
& \qquad\qquad\sideset{}{^*}\sum_{\chi(\text{mod}\ q)} \bigg|L \Big(\frac{1}{2} + \frac{A}{\log q}+it,\chi\Big) \bigg|^2 \bigg| L \Big( \frac{1}{2}+it+iv,\chi \Big) \bigg| dv dt.
\end{align*}
Observe that we have used positivity to extend the $t$-integral in $\mathcal{E}$ to all of $\mathbb{R}$. Our goal now is to show that
\begin{align}\label{bdonM}
\mathcal{M} \ll_{\rho,A,\delta} q (\log q)^{5/4}
\end{align}
and
\begin{align}\label{eq: cal E is smaller than J}
\mathcal{E} \leq \frac{1}{2}J \Big( \frac{A}{\log q}\Big).
\end{align}
Then \eqref{eq: J and J1}, \eqref{eq: upper bound for J with cal E}, \eqref{bdonM} and \eqref{eq: cal E is smaller than J} together give
\begin{align*}
J_1 \Big( \frac{A}{\log q}\Big)  &\ll_{\rho,A,\delta} q (\log q)^{5/4},
\end{align*}
and comparison with \eqref{eq: upper bound for M in terms of J} and \eqref{eq: J and J1} leads to
\begin{align*}
M(q) \ll_{\rho, A,\delta} q(\log q)^{9/4},
\end{align*}
as desired.

The bound \eqref{bdonM} follows from the observation that
\begin{align*}
\int_{-\log q}^{\log q} \left|W \left( \frac{1}{2} + \frac{A}{\log q} + it \right) \right|^6 dt \ll_{\rho,A} (\log q)^{-1}
\end{align*}
and the following lemma.

\begin{lemma}\label{lemmathird}
Let $t,v \in \mathbb{R}$ with $|t|\leq \log q$, and $s = 1/2 + A/\log q+it$. Then
\begin{align}\label{eq: upper bound on 3rd moment main term}
\sideset{}{^*}\sum_{\chi(\emph{mod}\ q)} \left|L \left( s,\chi\right) \right|^2 \bigg|\sum_{a \leq x} \frac{d_{1/2}(a) \chi(a)}{a^{s+iv}} \bigg|^2 \ll_{A} q(\log q)^{9/4},
\end{align}
and the implied constant is independent of $t$ and $v$. 
\end{lemma}
\begin{proof}
Applying Theorem \ref{mthm} with 
\[
\alpha=\frac{A}{\log q}+it,\qquad \beta=\frac{A}{\log q}-it\qquad\text{and}\qquad \alpha_a=\frac{d_{1/2}(a)}{a^{A/\log q+i(t+v)}},
\]
the quantity on the left side of \eqref{eq: upper bound on 3rd moment main term} is 
\begin{align*}
\ll_A q (\log q) \sum_{a,b \leq x} \frac{d_{1/2}(a)d_{1/2}(b)}{[a,b]}+O_\eps\big(q^{1-\eps}\big).
\end{align*}
Note we have used the fact that $d_{1/2}(n) \geq 0$. By an Euler product computation we find
\begin{align*}
\sum_{a,b \leq x} \frac{d_{1/2}(a)d_{1/2}(b)}{[a,b]} &\leq \prod_{p \leq x} \bigg( \sum_{i,j \geq 0} \frac{d_{1/2}(p^i)d_{1/2}(p^j)}{[p^i,p^j]} \bigg) = \prod_{p \leq x} \bigg( 1 + \frac{5}{4p} + O\Big( \frac{1}{p^2}\Big) \bigg) \\
&\ll (\log x)^{5/4} \ll (\log q)^{5/4},
\end{align*}
and the proof is complete.
\end{proof}

We now proceed to show \eqref{eq: cal E is smaller than J}. For notational simplicity we write
\begin{align*}
\mathcal{F} &= (4N)^N \frac{(yx^{-\delta})^{-A/\log q}}{(\delta \log x)^N}.
\end{align*}
We make the change of variables $z = t+v$, obtaining
\begin{align*}
\mathcal{E} = &\,\mathcal{F}\int_{-\infty}^{\infty} \int_{-\infty}^\infty \frac{\big| W (1/2 + A/\log q + it ) \big|^6}{\big(A^2/(\log q)^2 + (z-t)^2\big)^{(N+1)/2}} \\
&\qquad\qquad \sideset{}{^*}\sum_{\chi(\text{mod}\ q)} \bigg|L \Big(\frac{1}{2} + \frac{A}{\log q}+it,\chi\Big) \bigg|^2 \bigg|L \Big(\frac{1}{2} + iz,\chi\Big) \bigg| dz dt.
\end{align*}
By H\"older's inequality we obtain
\begin{align}\label{eq: bound for cal E}
\mathcal{E} &\leq \mathcal{F} \mathcal{E}_1^{2/3} \mathcal{E}_2^{1/3},
\end{align}
where
\begin{align*}
\mathcal{E}_1 =&\, \sideset{}{^*}\sum_{\chi(\text{mod}\ q)} \int_{-\infty}^{\infty}\bigg|W \Big(\frac{1}{2} + \frac{A}{\log q}+it\Big) \bigg|^6\bigg|L \Big(\frac{1}{2} + \frac{A}{\log q}+it,\chi\Big) \bigg|^3\\
&\qquad\qquad \int_{-\infty}^\infty \frac{dz}{\big(A^2/(\log q)^2 + (z-t)^2\big)^{(N+1)/2}}\,dt,
\end{align*}
and
\begin{align*}
\mathcal{E}_2 &=\sideset{}{^*}\sum_{\chi(\text{mod}\ q)}  \int_{-\infty}^\infty\left|L (\tfrac{1}{2} + iz,\chi) \right|^3 \int_{-\infty}^{\infty} \frac{\big| W (1/2 + A/\log q + it ) \big|^6}{\big(A^2/(\log q)^2 + (z-t)^2\big)^{(N+1)/2}}\,dtdz .
\end{align*}

It is easy to bound $\mathcal{E}_1$. By changing variables we obtain
\begin{align*}
\int_{-\infty}^\infty \frac{dz}{\big(A^2/(\log q)^2 + (z-t)^2\big)^{(N+1)/2}}&= \Big( \frac{\log q}{A}\Big)^N \int_{-\infty}^\infty \frac{du}{(1+u^2)^{(N+1)/2}} \leq \Big( \frac{\log q}{A}\Big)^N,
\end{align*}
and therefore
\begin{align}\label{eq: upper bound for cal E 1}
\mathcal{E}_1 &\leq\Big( \frac{\log q}{A}\Big)^N J \Big(\frac{A}{\log q}\Big).
\end{align}

Let us turn to $\mathcal{E}_2$. A change of variables yields
\begin{align*}
&\int_{-\infty}^\infty\frac{\big| W (1/2 + A/\log q + it ) \big|^6}{\big(A^2/(\log q)^2 + (z-t)^2\big)^{(N+1)/2}}\,dt\\
&\qquad\qquad= \Big( \frac{\log q}{A}\Big)^N \int_{-\infty}^\infty \frac{\left|W \left(1/2 + A/\log q + iz + iAu/\log q \right) \right|^6}{(1+u^2)^{(N+1)/2}}\,du.
\end{align*}
We wish to replace the argument in $W$ by $1/2 + A/\log q + iz$. Thus, we examine the quotient
\begin{align*}
&\frac{\left|W \left(1/2 + A/\log q + iz + iAu/\log q \right) \right|}{\left|W \left(1/2 + A/\log q + iz \right) \right|} \\
&\qquad\qquad= \left|\frac{q^{\rho(A/\log q + iz + iAu/\log q)}-1}{q^{\rho(A/\log q + iz)}-1} \right|  \left|\frac{A/\log q + iz}{A/\log q + iz + iAu/\log q} \right|.
\end{align*}
If $A \geq A_0(\rho)$, then
\begin{align*}
\left|\frac{q^{\rho(A/\log q + iz + iAu/\log q)}-1}{q^{\rho(A/\log q + iz)}-1} \right| \ll 1.
\end{align*}
Also, by considering the two cases $|Au/\log q - z| \geq |z|/3$ and $|Au/\log q - z| \leq |z|/3$, say, we find that
\begin{align*}
 \left|\frac{A/\log q + iz}{A/\log q + iz + iAu/\log q} \right|\ll 1 + |u|,
\end{align*}
and hence
\begin{align*}
&\int_{-\infty}^\infty \frac{\left|W \left(1/2 + A/\log q + iz + iAu/\log q \right) \right|^6}{(1+u^2)^{(N+1)/2}}\,du \\
&\qquad\qquad\ll \bigg|W \Big(\frac{1}{2} + \frac{A}{\log q}+iz\Big) \bigg|^6 \int_{-\infty}^\infty \frac{(1+|u|)^6}{(1+u^2)^{(N+1)/2}}du\ll \bigg|W \Big(\frac{1}{2} + \frac{A}{\log q}+iz\Big) \bigg|^6.
\end{align*}
We have therefore obtained
\begin{align}\label{eq: intermed bound for cal E 2}
\mathcal{E}_2 &\ll  \Big( \frac{\log q}{A}\Big)^N\sideset{}{^*}\sum_{\chi(\text{mod}\ q)} \int_{-\infty}^{\infty} \bigg|W \Big(\frac{1}{2} + \frac{A}{\log q}+iz\Big) \bigg|^6 \bigg|L \Big(\frac{1}{2} + iz,\chi\Big) \bigg|^3 dz.
\end{align}

The right side of \eqref{eq: intermed bound for cal E 2} is similar to $J\left(0 \right)$, but the argument of $W$ is perturbed. If we had precisely $J( 0)$, we could apply [\textbf{\ref{hb}}; Lemma 4] to relate $J(0)$ to $J(A/\log q)$. We claim the bounds
\begin{align}\label{eq: use convexity}
\int_{-\infty}^\infty &\left|W \Big(\frac{1}{2} + \frac{A}{\log q} + iz\Big) \right|^6 \sideset{}{^*}\sum_{\chi (\text{mod }q)} \left|L \Big(\frac{1}{2} + iz,\chi\Big) \right|^3 dz \nonumber \\ 
&\ll e^{3A/2 + O(\rho A)} \int_{-\infty}^\infty \left|W \Big(\frac{1}{2} + \frac{2A}{\log q} + iz\Big) \right|^6 \sideset{}{^*}\sum_{\chi (\text{mod }q)} \left|L \Big(\frac{1}{2} + \frac{A}{\log q} + iz,\chi\Big) \right|^3 dz \nonumber \\
&\leq e^{3A/2 + O(\rho A)} J\Big(\frac{A}{\log q} \Big). 
\end{align}
The proof is very similar to the proof of [\textbf{\ref{hb}}; Lemma 4], so we do not give it. The only real difference is that we must do more careful bookkeeping with the constants. To prove \eqref{eq: use convexity} we also use the bounds
\begin{align*}
\left|W \Big(\frac{1}{2} + \frac{2A}{\log q} + iz\Big) \right| &\ll e^{O(\rho A)} \left|W \Big(\frac{1}{2} + \frac{A}{\log q} + iz\Big) \right|, \\
\left|W \Big(\frac{1}{2} + iz\Big) \right| &\ll e^{O(\rho A)} \left|W \Big(\frac{1}{2} + \frac{A}{\log q} + iz\Big) \right|,
\end{align*}
which hold for $A \geq A_0(\rho)$. By \eqref{eq: intermed bound for cal E 2} and \eqref{eq: use convexity} we therefore obtain
\begin{align}\label{eq: bound for cal E 2}
\mathcal{E}_2 &\ll  e^{3A/2 + O(\rho A)}\Big( \frac{\log q}{A}\Big)^N J \Big(\frac{A}{\log q} \Big).
\end{align}

We compare \eqref{eq: bound for cal E}, \eqref{eq: upper bound for cal E 1} and \eqref{eq: bound for cal E 2} and find
\begin{align*}
\mathcal{E} &\ll  \mathcal{F} e^{A/2+ O(\rho A)}\Big( \frac{\log q}{A}\Big)^N J \Big(\frac{A}{\log q} \Big).
\end{align*}
Recalling our definitions, we have
\begin{align*}
\mathcal{E} &\leq C e^{-(1+O(\rho \delta^{-1}))\delta A}\Big(\frac{8N}{\delta A} \Big)^N   J \Big(\frac{A}{\log q} \Big)
\end{align*}
for some absolute constant $C > 0$. We obtain \eqref{eq: cal E is smaller than J} by choosing $\rho$ to be sufficiently small in terms of $\delta$, and then choosing $A$ to be sufficiently large in terms of $\rho$ and $\delta$.

\section*{Acknowledgments}

The second author was supported by NSF grant DMS-1501982. The authors would like to thank Sandro Bettin and Maksym Radziwi\l\l\, for various helpful comments.

\end{document}